\documentclass{lms}
\usepackage{amssymb}

\newtheorem{theorem}{Theorem}[section] 
\newtheorem{lemma}[theorem]{Lemma}     
\newtheorem{corollary}[theorem]{Corollary}
\newtheorem{proposition}[theorem]{Proposition}

\newnumbered{assertion}{Assertion}    
\newnumbered{conjecture}{Conjecture}  
\newnumbered{definition}{Definition}
\newnumbered{hypothesis}{Hypothesis}
\newnumbered{remark}{Remark}
\newnumbered{note}{Note}
\newnumbered{observation}{Observation}
\newnumbered{problem}{Problem}
\newnumbered{question}{Question}
\newnumbered{algorithm}{Algorithm}
\newnumbered{example}{Example}
\newunnumbered{notation}{Notation} 

\providecommand\ve{\varepsilon}
\providecommand\rr{{\mathbf R}^2}
\providecommand\tra{\mathrm{tr}}
\def\tr{\mathrm{tr}}
\def\trace{\hskip-1pt\mathrel{\hbox{\vrule height 7pt depth 0pt width .3pt%
\vrule height .3pt depth 0pt width 6pt}}}

\providecommand{\aaa}{\mathcal{A}}
\providecommand{\bbb}{\mathcal{B}}

\title[BAP for Sobolev spaces]{Bounded Approximation Property for Sobolev spaces on simply-connected planar domain}%
\author{Maria Roginskaya, Micha\l{} Wojciechowski}%

\date{}

\begin{document}

\maketitle
\begin{abstract}
We show that Sobolev space $W^1_1(\Omega)$ of any planar one-connected domain $\Omega$ has the Bounded Approximation property. The result holds independently from the properties of the boundary of $\Omega$. The prove is based on a new  decomposition of a planar domain.
\end{abstract}

\section{Introduction}
In the present paper we study the bounded approximation property (BAP) of Sobolev spaces $W^1_1(\Omega)$ of functions with integrable gradient on a bounded connected simply connected planar domain $\Omega$.

Recall that a Banach space $X$ has the {\it bounded
approximation property} if for every $\ve>0$ and every finite set
$S\subset X$ there exists a finite rank operator $T_\ve:X\to X$ with
norm bounded independent of $\ve$ and $S$, such that $\|T_\ve x-x\|<\ve$ for all
$x\in S$.

\subsection{Sobolev spaces and $BV$} Let $\Omega\subset \rr$ be an open domain and $E$ be a linear
(finite dimensional) space. By $M(\Omega, E)$ we denote the space
of bounded $E$-valued measures on $\Omega$ with norm
$\|\mu\|_M=|\mu|(\Omega)$ where
$|\mu|$ denotes the total variation of the measure $\mu$
(cf. [DU], Sect. I.1, Def. 4).

By $BV(\Omega)$ we denote the space of all measurable functions summable on $\Omega$,
whose distributive first order derivatives are measures from $M(\Omega,\rr)$. The norm on $BV(\Omega)$ is defined as
$$
\|f\|_{BV}=\|f\|_1+ |\nabla f|(\Omega);
$$
here the symbol $\nabla f$ denotes the distributional gradient of $f$.
By $W^1_{1}(\Omega)$ we denote the subspace of
$BV(\Omega)$ of those functions whose derivatives are absolutely
continuous with respect to the 2-dimensional Hausdorff measure on $\Omega$.

Our main result is the following.

\begin{theorem}\label{main}
For every simply connected bounded planar domain $\Omega$ the space $W^1_1(\Omega)$ has BAP.
\end{theorem}

The BAP property of $W^1_1(\Omega)$ has been already proved for some special classes of the domains (cf. [ACPP]). Typically the requirement on the domain includes some sort of smoothness of its boundary. As far to our knowledge this is a first result of this type which poses no conditions on the boundary of the domain.

The main ingredient of the proof of the above statement is the following geometrical construction.

\subsection{Geometric decomposition}

\begin{definition} We say that a set $A$ whose boundary is a simple curve consisting of intervals and circular arc is {\it $\eta$-elementary}\footnote{Notice that domains with this property are Lavrentiev domains, but as we need explicit value of the constant involved, it seems as a less of mouthful to call such domains $\eta$-elementary} when for any pair $x,y\in \partial A$ there exists a rectifiable curve $\gamma\subset \partial A$ connecting $x$ with $y$ such that $|\gamma|\leq \eta\ d(x,y)$.
\end{definition}
(We denote by $d(x,y)$ the Euclidean distance between $x$ and $y$, and by $|\gamma|$ the length of a rectifiable curve $\gamma$.)

\begin{theorem}[(Partition Theorem)]\label{partition} There exists four universal constants $k\in {\mathbf N}$ and $c,C,\eta>0$, such that following holds for every bounded connected simply connected domain $\Omega\subset \rr$. For every $\ve>0$ there exist
two finite families $\aaa_\ve$ and $\bbb_\ve$ of open subsets of $\Omega$,
such that:
\begin{enumerate}
\item\label{dec.Aclosed} $A$ is an $\eta$-elementary domain for every $A\in
\aaa_\ve$;

\item\label{dec.diam} For every $A\in\aaa_\ve$ holds $c\ve\leq {\rm diam}(A)\leq C\ve$;

\item\label{dec.disjoint} every two different sets from $\aaa_\ve\cup \bbb_\ve$ are disjoint;

\item\label{dec.fewA} for every $A\in \aaa_\ve$ there are at most  $k$ different
$A'\in \aaa_\ve$ such that $\overline{A}\cap \overline{A'}\cap\Omega\neq \emptyset$;

\item\label{dec.oneB} for every $A\in \aaa_\ve$ there is at most one
$B\in \bbb_\ve$ such that $\overline{A}\cap \overline{B}\cap\Omega \neq\emptyset$;

\item\label{dec.oneA} for every $B\in \bbb_\ve$ there is exactly one $A=A_B\in\aaa_\ve$
such that $\overline{B}\cap \overline{A_B}\cap\Omega \neq \emptyset$;

\item\label{dec.dec} $\bigcup_{A\in\aaa_\ve}\overline{A}\cup\bigcup_{B\in\bbb_\ve}\overline{B} \supset \Omega$;

\item\label{dec.approx} $(\bigcup\limits_{B\in\bbb_\ve} B) \cup (\bigcup\limits_{B\in \bbb_\ve}A_B)\subset (\rr\setminus \Omega)_\ve$.
\end{enumerate}
\end{theorem}

Partition Theorem is of its own interest. It presents decomposition of $\Omega$ into "good" pieces $\aaa_\ve$ and "bad" $\bbb_\ve$ which are small and isolated. The cardinalities of $\aaa_\ve$ and $\bbb_\ve$ as  functions of $\ve$ could be used to quantitative description of how "bad" the shape of a simply connected planar domain is. Connections of this measure of badness with the others, such as for example dimension of the boundary is unknown to us. We postpone the proof of the Partition Theorem to section \ref{decomposition}.

\subsection{Notation and organization}
In our construction we often need to use interior/closure of the same set. To avoid a confusion we denote the closure of a set $A$ as $\overline{A}$, the interior as $A^\circ$, and the boundary as $\partial A$. On a few occasion where it doesn't lead to a confusion we write $A$ for $\overline{A}$ or $A^\circ$.

For any set $K$ the set $K_\ve$ stands for $\{x\in \rr: dist(x,K)<\ve\}$.

If $A\subset \Omega$ is a measurable
subset and $\mu\in M(\Omega, E)$, by
$\mu\trace A$ we denote the measure defined by
$(\mu\trace A)(B)=\mu(A\cap B)$. By
$\lambda_n$ we denote the $n$-dimensional Hausdorff measure ($n=1,2$).

By $\gamma$ we denote a simple curve in the plane, and if it is important to us that it has ends  $a$ and $b$ we denote it $\gamma(a,b)$. It is
a topological fact (Jordan lemma) that if we have two curves $\gamma_1(a,b)$ and
$\gamma_2(a,b)$, such that $\gamma_1\cap\gamma_2=\{a,b\}$, then those
two curves together is the boundary of a bounded connected domain, which we denote
$\Omega_{\gamma_1,\gamma_2}$. 


The organization of the paper is as follows. In Section \ref{mainproof} we prove Theorem \ref{main} assuming
Partition Theorem. Proof  uses mainly trace estimate due to Galiardo together with the
Poincare inequality for regular domains.

In Section \ref{decomposition} we give the proof of Partition Theorem \ref{partition}.

\section{Proof of Theorem \ref{main}}\label{mainproof}

\subsection{Averaging Operator}
An implication of the Partition Theorem is the boundedness of the averaging operator defined below.

Let $f\in W^1_1(\Omega)$ and
$A\in \aaa_\ve$. We use the notation $f_A=\lambda_2(A)^{-1}\int_A f\,d\lambda_2$.
By $\mathrm{tr}_A f$ we denote the Galiardo trace of function $f\in W^1_1(A)$ or $f\in BV(A)$ on the boundary $\partial A$ (here and everywhere else in the article we refer to \cite{Z} for standard facts about the Galiardo trace).

Note that since the boundary of an elementary domain $A$ is a union of intervals and circular arcs, the definition of the trace operator is identical to the definition of the trace on a line.

If $\Lambda\subset \partial A$ is a $\lambda_1$-measurable
set of positive $\lambda_1$ measure we put $g_\Lambda=\lambda_1(\Lambda)^{-1}\int_\Lambda g\,d\lambda_1$ for every function $g$ integrable on $\Lambda$.

For $B\in \bbb_\ve$ we set
$ \Lambda(B)=\partial B\cap\partial A_B\cap\Omega$. Note that on $\Lambda(B)$ we have
$\tra_Bf_{|B}=\tra_{A_B}f_{|A_B}$ for $f\in W^1_1(\Omega)$, as the left and right traces coincide inside the domain (see \cite{Z}).

For $f\in BV(\Omega)$ we put

$$
Tf(x)=\left\{\begin{array}{ll}
f_A &\mbox{ for $x\in \mathrm{int} A$}\\
(\tr_B f)_{\Lambda(B)}\cdot h_B(x) &\mbox{for $x\in
\mathrm{int} B$}\end{array}
                     \right.
\eqno{(1)}
$$
 where $h_B$ is any fixed for the given decomposition function from  $BV(B)$, chosen
 so that
 \begin{eqnarray*}
 \|h_B\|_{BV(B)}+\int_{\Lambda(B)}&|\tr_B h_B-1|\,d\lambda_1\\
 &\leq 2\inf_{g\in BV(B)}\big(\|g\|_{BV(B)}
 +\int_{\Lambda(B)}|\tr_B g-1|\,d\lambda_1\big).
 \end{eqnarray*}

\begin{theorem}\label{T:WtoBV} For any decomposition of $\Omega$ provided by the Partition theorem the operator $T: W^1_1(\Omega)\to BV(\Omega)$ defined by (1) is a bounded
operator with norm depending only on $k$ and $\eta$.
\end{theorem}

In order to prove Theorem \ref{T:WtoBV} we need several lemmas.

\begin{lemma}\label{L:traceestimate} Let $A\subset \rr$ is $\eta$-elementary domain of diameter $1$.
Then, there exists $C=C(\eta)>0$ such that for every
$f\in BV(A)$,
$$
\|\tr_A f\|_{L_1(\partial A)}\leq C(\|\nabla
f\|_{M(A)}+\|f\|_{L^1(A)}).
$$
\end{lemma}

\begin{proof} This would be a particular case of Theorem 5.10.7 in \cite{Z}, except we want the constant depend only from $\eta$, but not otherwise from the shape of $A$.

We show how to derive the above estimate from a similar bound for the
unit disk (for which the constant is fixed).

The definition of an $\eta$-elementary domain implies that there exists $\eta$-bi-Lipschitz function
$\psi:\mathbf T \to \partial A\subset \rr$. By theorems 7.9,7.10 in \cite{Pom},
it has $M$-bi-Lipschitz extension $\psi:\rr\mapsto\rr$, which in particular maps
the unit disc $D$ to $A$, with bi-Lipschitz constant $M$ depending only from $\eta$.
Let $T: BV(A)\to BV(D)$ be defined by $Tf(x)=f\circ\psi(x)$ for $x\in D$.
Then, denoting by  $J\psi$ a determinant of the
Jacobi matrix of $\psi$, we get

\begin{eqnarray*}
\|Tf\|_{L^1(D)} =&\int_D |f\circ\psi(x)|\,d\lambda_2(x)\\
= &\int_A |f(y)|\cdot |J\psi^{-1}(y)|\,d\lambda_2(y)
\leq & M^2\|f\|_{L^1(A)},
\end{eqnarray*}
and
$$\|\nabla(Tf)\|_{M(D)}= |\nabla (f\circ \psi)|(D)
\leq  M\cdot|(\nabla f)\circ\psi|(D)\\ =M\cdot |\nabla f|(A)= M\cdot\|\nabla f\|_{M(A)}
\eqno{(2)}
$$
Since the composition with a bi-Lipschitz map $\psi$ is an isomorphism from $L^1(\partial A)$ onto $L^1(\mathbf T)$ with norm depending on $\eta$ only for the traces, we see that
$$
\|\tr_A f\|_{L_1(\partial A)}\leq C(\eta)\|\tr_{\mathbf T} Tf\|_{L^1(\mathbf T)}\leq C C(\eta) (\|\nabla
Tf\|_{M(D)}+\|Tf\|_{L^1(D)})\leq$$ $$ CM^2 C(\eta)(\|\nabla
f\|_{M(A)}+\|f\|_{L^1(A)}),$$
thus proving the lemma.

\end{proof}

\begin{lemma}\label{L:oscilationestimate} Let $A\in \rr$ be an $\eta$-elementary domain with diameter $1$.
Then, there exists $C=C(\eta)>0$ such that for every $f\in BV(A)$,
$$
\|f-f_A\|_{L^1(A)}\leq C\|\nabla f\|_{M(A)}.
$$
\end{lemma}
\begin{proof}
This is essentially the Poincare lemma. The independence of the constant $C$
from anything but $\eta$ is a well known fact, which is proven for example in \cite{Mar}. We give a proof for the sake of completeness.

Consider the space $\dot{BV}$, which is a space of all measurable locally summable functions whose distributive first order derivative is a finite measure. The norm on $\dot{BV}$ is defined by $\|f\|_{\dot{BV}}=\|\nabla f\|_M$, with the corresponding factorization over the constants.

Let $\psi$ be the mapping from the previous lemma and $T$ the corresponding operator $Tf=f\circ\psi$, $T:\dot{BV}(A)\mapsto \dot{BV}(D)$. By (2) we see that the norm of the operator $T$ is bounded by a constant which depends only from $\eta$.

Denote by $\widetilde{T}^{-1}: \dot{BV}(\rr)\to \dot{BV}(\mathbf\rr)$ given by $\widetilde{T}^{-1}g=g\circ\psi^{-1}$. In the similar fashion we see that the norm of $\widetilde{T}^{-1}$ is bounded by a constant which depends only from $\eta$.

Let further $S: \dot{BV}(D)\to \dot{BV}(\rr)$ denote an extension operator defined as
$$S(f)=\left\{
    \begin{array}{ll}
      f-f_D, & \hbox{ on $D$;} \\
      0, & \hbox{on $\rr\setminus D$.}
    \end{array}
  \right.
$$
For this operator, $\|Sf\|_{\dot{BV}(\rr)}\leq \|Sf\|_{BV(\rr)}\leq C\|f-f_D\|_{BV(D)}\leq C_1\|f\|_{\dot{BV}(D)}$, the second inequality being particular case of Lemma 5.10.4 in \cite{Z}, and the last follows to the classical Poincare inequality for the unit disc.

Finally, let $R:  \dot{BV}(\rr)\to \dot{BV}(2D)$ the operator of restriction to doubled disc, which is obviously has the norm $1$.

We consider the extension operator
$$
Q:= R\circ\widetilde{T}^{-1}\circ S\circ T: \dot{BV}(A)\to \dot{BV}(2D)
$$
whose norm depends from $\eta$ only.

Now we use the Theorem 5.12.7 of \cite{Z}.

\begin{theorem*} Let $\Omega$ be a connected admissible domain in $\mathbb{R}^n$ and let $\mu$ be a non-trivial positive Radon measure such that $\rm{spt}\mu\subset\overline\Omega$ and for some constant $M>0$ that $\mu(B(x,r))\leq M r^{n-1}$ for all balls $B(x,r)$ in $\mathbb{R}^n$. Then there exists a constant $C=C(\Omega)$ such that for each $u\in BV(\Omega)$, $$\|u-\overline{u}\|_{L^2(\Omega)}\leq C\frac{M}{\mu(\overline{\Omega})}\|\nabla u\|_{M(\Omega)},$$
where $\overline{u}=\frac1{\mu(\overline{\Omega})}\int u(x)d\mu(x)$.
\end{theorem*}

We apply the Theorem with $\Omega=2D$ and $\mu=\lambda_2\trace A$ (wlog we may assume $A\subset 2D$).

It is not difficult to see that $\mu(B(x,r))\leq Cr$ for every $x\in \rr$, $r>0$ and some absolute constant $C>0$. Thus, as $c(\eta)\leq \lambda_2(A)\leq C(\eta)$, the theorem claims that for $g\in BV(2D)$ holds,
$$\|g-g_A\|_{L^1(2D)}\leq \sqrt{C_1}\|g-g_A\|_{L^2(2D)}\leq C_2\|\nabla g\|_{M(2D)},$$
where the constant $C_2$ depends only from $\lambda_2(A)$ and, thus, has an upper bound which depends only from $\eta$.

It remains to observe that $\dot{BV}(2D)=BV(2D)$ (though the identification is not bounded) and that for $f\in BV(A)$ holds $(Qf)\trace A=f-c_f$. Thus, $$\|f-f_A\|_{L^1(A)}=\|Qf-(Qf)_A\|_{L^1(A)}\leq  \|Qf-(Qf)_A\|_{L^1(2D)}$$
$$\leq C_2\|\nabla(Qf)\|_{M(2D)}=C_2\|Qf\|_{\dot{BV}(2D)}\leq C_2\|Q\|\cdot\|f\|_{\dot{BV}(A)}=C_3\|\nabla f\|_{M(A)}.$$

\end{proof}

Combining Lemma \ref{L:oscilationestimate} and Lemma \ref{L:traceestimate}
we get following.

\begin{lemma}\label{xyz} Let $A\subset \rr$ is $\eta$-lipschitz domain.
Then, there exists $C=C(\eta)>0$ such that for every $f\in BV(A)$,
$$
\|\tr_A f-f_A\|_{L^1(\partial A)}\leq C\|\nabla f\|_{M(A)}.
$$
\end{lemma}
\begin{proof} As both parts re-scale of order $1$ when one changes the size of $A$, we may assume
that diameter of $A$ is $1$.

Consider $g=f-f_A$. Using Lemma \ref{L:traceestimate}
we see that
\begin{eqnarray*}
\|\tr_A f-f_A\|_{L^1(\partial A)}&=\|\tr_A g\|_{L_1(\partial A)}\\ &\leq C(\|\nabla
g\|_{M(A)}+\|g\|_{L^1(A)})\\&=C(\|\nabla
f\|_{M(A)}+\|f-f_A\|_{L^1(A)}).
\end{eqnarray*}

Applying now Lemma \ref{L:oscilationestimate} to the second term
on the right side we derive the desired conclusion.
\end{proof}

Next three lemmas will be used to  prove of the
boundedness of the averaging operator.

\begin{lemma}\label{abc} There exists $C=C(\eta)>0$ with the following property.
 Let $A, A'\in \aaa_\ve$
and $\Lambda=\partial A\cap\partial A'  \cap \Omega$. Let
$D_{A,A'}=(\overline{A\cup A'}\cap \Omega)^\circ$.
Then for every $f\in W^1_1(D_{A',A'})$,
$$
\int_{\Lambda}|f_A-f_{A'}|\,d\lambda_1\leq C\|\nabla
f\|_{L^1(D_{A,A'})}.
$$
\end{lemma}
\begin{proof} For $f\in W^1_1(D_{A,A'})$ on $\Lambda$ holds $\rm{tr}_A f=\rm{tr}_{A'}f$. Using first the triangle inequality and then Lemma \ref{xyz} we get
\begin{eqnarray*}
\int_{\Lambda}|f_A-f_{A'}|\,d\lambda_1 &\leq
\int_{\Lambda}|\tr_A f-f_A|\,d\lambda_1  +\int_{\Lambda}|\tr_{A'}
f-f_{A'}|\,d\lambda_1 \\
&\leq  \int_{\partial A}|\tr_A f-f_A|\,d\lambda_1 +\int_{\partial
A'}|\tr_{A'} f-f_{A'}|\,d\lambda_1\\
&\leq C \|\nabla f\|_{L^1(A)} + C\|\nabla f\|_{L^1(A')}\\
&\leq  C\|\nabla f\|_{L^1(D_{A,A'})}
\end{eqnarray*}
\end{proof}

The following two lemmas will be used to estimate the behavior of the averaging operator $T$ near the common boundary of two pieces of which only one is $\eta$-Lipschitz.

\begin{lemma}\label{pqr} There exists $C=C(\eta)>0$ such that for every $B\in\bbb_\ve$  and
$f\in W^1_1(\Omega)$,
$$
\int_{\Lambda(B)}|(\tr_{B} f)_{\Lambda(B)}-f_{A_B}|\,d\lambda_1\leq C\|\nabla
f\|_{L^1(A_B)},
$$
where $\Lambda(B)=\partial B\cap\partial A_B\cap\Omega$.
\end{lemma}

\begin{proof} Again, as $f\in W^1_1(\Omega)$, we know that $\rm{tr}_B f=\rm{tr}_{A_B} f$ on $\Lambda(B)$. So,
\begin{eqnarray*}
\int_{\Lambda(B)}|(\tra_B f)_{\Lambda(B)}-f_{A_B}|\,d\lambda_1 &=
 \int_{\Lambda(B)}|(\tra_{A_B} f)_{\Lambda(B)}-f_{A_B}|\,d\lambda_1\\
& \mbox{(notice that the integrand is a constant)}\\
&=\lambda_1(\Lambda(B))\cdot |(\tr_{A_B} f)_{\Lambda(B)}-f_{A_B}|\\
&=\lambda_1(\Lambda(B))\cdot |(\tra_{A_B} f-f_{A_B})_{\Lambda(B)}|\\
  &=\big|\int_{\Lambda(B)}(\tra_{A_B} f-f_{A_B})\,d\lambda_1\big|\\
  &\leq \int_{\Lambda(B)}|\tra_{A_B} f-f_{A_B}|\,d\lambda_1\\
  &\leq \int_{\partial A_B}|\tr_{A_B} f-f_{A_B}|\,d\lambda_1\\
 \mbox{(which by the Lemma \ref{xyz})}&\leq C\|\nabla f\|_{L^1(A_B)}.
\end{eqnarray*}

\end{proof}

\begin{lemma}\label{L27}
There exists $C=C(\eta)>0$ such that for every $B\in\bbb_\ve$ and $f\in W^1_1(\Omega)$,
\begin{eqnarray*}
|(\tra_B f)_{\Lambda(B)}|\cdot\big( \|h_B\|_{BV(B)}&+\int_{\Lambda(B)}|\tra_B h_B-1|\,d\lambda_1\big)\\
&\leq C\big(\|f\|_{BV(B)}+\|f\|_{BV(A_B)}\big).
\end{eqnarray*}
\end{lemma}

\begin{proof}
 Using the definition of $h_B$, and the fact that, as $f\in W^1_1(\Omega)$, $\rm{tr}_B=\rm{tr}_{A_B}$ on $\Lambda(B)$, we see that
\begin{eqnarray*}
|(\tra_B f)_{\Lambda(B)}|\cdot\big( \|h_B\|_{BV(B)}+\int_{\Lambda(B)}|\tra_B h_B-1|\,d\lambda_1\big)\\
\leq 2|(\tr_Bf)_{\Lambda(B)}| \inf_{g\in BV(B)}\big(\|g\|_{BV(B)}
+\int_{\Lambda(B)}|\tr_Bg-1|\,d\lambda_1\big)\\
\mbox{(when $\tilde{g}=|(\tr_Bf)_{\Lambda(B)}|g$})\\
=2 \inf_{\tilde{g}\in BV(B)}\big(\|\tilde{g}\|_{BV(B)}
+\int_{\Lambda(B)}|\tr_B\tilde{g}-(\tr_Bf)_{\Lambda(B)}|\,d\lambda_1\big)\\
\leq 2\|f\|_{BV(B)}
+2\int_{\Lambda(B)}|\tr_Bf-(\tr_Bf)_{\Lambda(B)}|\,d\lambda_1\\
\mbox{(and by a triangle inequality)}\\
\leq
2\|f\|_{BV(B)}+2\int_{\Lambda(B)}|f_{A_B}-(\tr_{B}f)_{\Lambda(B)}|\,d\lambda_1
 +2\int_{\Lambda(B)}|\tr_{A_B}f-f_{A_B}|\,d\lambda_1.
\end{eqnarray*}
Estimating now the first integral by Lemma \ref{pqr} and the second by Lemma \ref{xyz}
we get the desired bound.
\end{proof}

\begin{proof}[of Theorem \ref{T:WtoBV}] Let $f\in W^1_1(\Omega)$. Then
\begin{eqnarray}\label{ast}
\nabla (Tf) &=  \sum_{A,A'\in\aaa_\ve} \nabla(Tf)\trace(\partial
A\cap\partial A')
\end{eqnarray}
$$+ \sum_{B\in\bbb_\ve}\nabla(Tf)\trace \Lambda(B) +\sum_{B\in\bbb_\ve}\nabla(Tf)\trace B.$$

If $A,A'\in\aaa_\ve$ and $\Lambda=\partial A\cap\partial A'\cap \Omega\neq \emptyset$ it is easy to see that $\nabla Tf$ is a one-dimensional vector valued measure for which
by Lemma \ref{abc} holds,

$$
 \|\nabla(Tf)\trace(\partial A\cap\partial
A')\|_{M(\Omega)}=\int_{\Lambda}|f_A-f_{A'}|\,d\lambda_1 \leq
C\|\nabla f\|_{L^1(A\cup A')}.
$$
If $\partial A\cap\partial A'\cap \Omega= \emptyset$ the corresponding term vanishes.

Summing over all neighboring pairs $A, A'\in \aaa_\ve$, and taking
into account that by (iii) every $A\in \aaa_\ve$ has at most $k$
different neighbors  $A'\in \aaa_\ve$ we get for the first sum in
(\ref{ast})
$$
\|\sum_{A,A'\in\aaa} \nabla(Tf)\trace(\partial A\cap\partial
A')\|_{M(\Omega)}\leq C\cdot k\cdot\|\nabla f\|_{L^1(\Omega)}\leq C'\|f\|_{W^1_1(\Omega)}.
$$
For $B\in \bbb_\ve$ we get by the triangle inequality
\begin{eqnarray*}
\|\nabla(Tf)\trace \Lambda(B)\|_{M(\Omega)}
&=\int_{\Lambda(B)}|(\tr_Bf)_{\Lambda(B)}\cdot\tr_B
h_B - f_{A_B}|\,d\lambda_1\\
&\leq \int_{\Lambda(B)}|(\tr_Bf)_{\Lambda(B)}\tr_B h_B-(\tr_Bf)_{\Lambda(B)}|\,d\lambda_1\\
&\qquad + \int_{\Lambda(B)}|(\tr_B f)_{\Lambda(B)}-f_{A_B}|\,d\lambda_1
\end{eqnarray*}
By Lemma \ref{pqr}, the second integral on the right hand side is bounded by
$C\|\nabla f\|_{L^1(A_B)}$. By Lemma \ref{L27} the first summand is bounded by
$C\big(\|f\|_{BV(B)}+\|f\|_{BV(A_B)}\big)$. Hence
$$
\|\nabla(Tf)\trace \Lambda(B)\|_{M(\Omega)}\leq
C\big(\|f\|_{BV(B)}+\|f\|_{BV(A_B)}\big).
$$
Summing over $B\in \bbb_\ve$ and taking into account that all $A_B$ are distinct we get for the second sum in \ref{ast}
$$
\|\sum_{B\in\bbb_\ve}\nabla(Tf)\trace \Lambda(B)\|_{M(\Omega)}
\leq C\cdot \| f\|_{W^1_1(\Omega)}.
$$
For the third sum in (\ref{ast}) we have
\begin{eqnarray*}
 \|\sum_{B\in\bbb_\ve}\nabla(Tf)\trace
B\|_{M(\Omega)}&=\sum_{B\in\bbb_\ve}\|\nabla(Tf)\trace
B\|_{M(\Omega)}\\
&=\sum_{B\in\bbb_\ve}\|\nabla h_B\|_{M(B)}\cdot|(\tr_Bf)_{\Lambda(B)}|
\end{eqnarray*}
By Lemma \ref{L27} the last sum is dominated by
$$
C\sum_{B\in\bbb_\ve}\big(\|f\|_{BV(B)}+\|f\|_{BV(A_B)}\big) \leq C \|f\|_{W^1_1(\Omega)}.
$$
Summing all the three estimates shows $\|\nabla(Tf)\|_{M(\Omega)}\leq C\|f\|_{W^1_1(\Omega)}$. The estimate $\|Tf\|_1\leq C \|f\|_{W^1_1(\Omega)}$ is trivial, which completes the proof of Theorem \ref{T:WtoBV}.
\end{proof}

\subsection{Proof of Theorem \ref{main}}

Let $E\subset W^1_1(\Omega)$ be a fixed finite dimensional subspace and $\delta>0$.
Let $\ve>0$ be such that for every $f\in E$
$$
\int_{\Omega\cap(\rr\setminus\Omega)_\ve}|f|+|\nabla f|\,d\lambda_2<\delta\cdot\|f\|_{W^1_1(\Omega)}
$$
Let $\aaa_\ve$, $\bbb_\ve$ be the families of sets provided by Partition Theorem. Let $T_\ve:W^1_1(\Omega)\to BV(\Omega)$ be the operator given by (1). Let further $\Xi_\ve=\big(\bigcup_{A\in\aaa_\ve}\overline{A}\cap\Omega\big)^\circ$ and $V_\ve:W^1_1(\Omega)\to BV(\Omega)$ be the composition
$$
V_\ve=R_\ve\circ (Id-T_\ve),
$$
where $R_\ve: BV(\Omega)\to BV(\Omega)$ is a restriction operator $f\mapsto f{\mathbf 1}_{\Xi_\ve}$.

Despite the fact that $R_\ve$ itself may not be bounded independently from $\ve$,
we can prove that $V_\ve$ does. Indeed, the image of $R_\ve\circ (Id-T_\ve)$ is contained in the space
$$
Y_\ve=\{ f\in BV(\Omega): f_A=0\quad\mbox{ for $A\in\aaa_\ve$, $f|_{A}\in W^1_1(A)$, and $f|_{B}\equiv 0$ for $B\in\bbb_\ve$} \}
$$
If $f\in Y_\ve$ then
\begin{eqnarray*}
\|f\|_{BV(\Omega)}&=\|f\|_{L^1(\Omega)}+\sum_{A\in\aaa_\ve} |\nabla f|(A)+
\sum_{A,A'\in \aaa_\ve} |\nabla f|(\partial A\cap\partial A'\cap\Omega)\\
&\qquad\qquad\qquad\qquad+\sum_{B\in\bbb_\ve}|\nabla f|(\Lambda(B))\\
&=\|f\|_{L^1(\Omega)}+\sum_{A\in\aaa_\ve} |\nabla f|(A)+
\sum_{A,A'\in \aaa_\ve} \int_{\partial A\cap\partial A'\cap\Omega}|\tra_A f-\tra_{A'}f|\,d\lambda_1\\
&\qquad\qquad\qquad\qquad+
\sum_{B\in\bbb_\ve} \int_{\Lambda(B)}|\tra_{A_B}f|\,d\lambda_1\\
&\leq \|f\|_{L^1(\Omega)}+\sum_{A\in\aaa_\ve} |\nabla f|(A)+
\sum_{A\in \aaa_\ve} \int_{\partial A}|\tra_A f|\,d\lambda_1.
\end{eqnarray*}
The last inequality follows from $\bigcup_{A,A'\in\aaa_\ve} (\partial A\cap\partial A'\cap\Omega)\cup \bigcup_{B\in\bbb_\ve} \Lambda(B)=\bigcup_{A\in\aaa_\ve}\partial A$, with everything but a set of $\lambda_1$-measure zero counted different amount of times on both sides.

Remembering that $f_A=0$ for $A\in \aaa_\ve$ and applying Lemma \ref{xyz} we get for $f\in Y_\ve$,
$$
\|f\|_{BV(\Omega)}\leq \|f\|_{L^1(\Omega)}+C\sum_{A\in\aaa_\ve} |\nabla f|(A),
\eqno{(3)}
$$
where the constant $C$ does not depend from $\ve$, but only from the universal constants $k$ and $\eta$.

Now the boundedness of $R_\ve\circ(Id-T_\ve)$ follows from the relations
$$
\|R_\ve\circ(Id -T_\ve)f\|_{L^1(\Omega)}\leq \|f\|_{L^1(\Omega)}+\|T_\ve f\|_{L^1(\Omega)}
$$
where $T_\ve$ is bounded and
$$
|\nabla(Id-T_\ve)f|(A)=|\nabla f|(A)\qquad\mbox{for $A\in \aaa_\ve$}.
$$

Another corollary of formula (3) is the following observation. Let
$X_\ve=\big(\bigoplus_{A\in \aaa_\ve}{W}^1_{1,0}(A)\big)_{\ell_1}$, where $W^1_{1,0}(A)=\{f\in W^1_1(A):f_A=0\}$. The norm of
$(g^A)_{A\in\aaa_\ve}\in X_\ve$ is given by
$$
\|(g^A)\|_{X_\ve}=\sum_{A\in\aaa_\ve}\|g^A\|_{L^1(A)}+\sum_{A\in\aaa_\ve} |\nabla g^A|(A).
$$
By (3) the space $Y_\ve$ is isomorphic to $X_\ve$ with the norm of the isomorphism dependent only $\eta$ and not $\ve$, when the isomorphism $J_\ve: Y_\ve\mapsto X_\ve$ is given by
$$
J_\ve: f\mapsto (f_{|A})_{A\in\aaa_\ve}\in X_\ve.
$$

Obviously $W^1_{1,0}(A)$, as a complemented subspace of $W^1_1(A)$ of co-dimension one, is isomorphic to the latter with norm non exceeding 2. By \cite{St} Thm 5 Chapt.  VI, there exists an extension operator $Ext: W^1_1(A)\to W^1_1(\mathbb{R}^2)$ with norm depending only on the bi-Lipschitz constant of $\partial A$. It follows from \cite{PW1} Thm 10 (see also \cite{PW2}) that $W^1_{1,0}(A)$ is isomorphic to $W^1_1(\mathbb{R}^2)$ with uniformly bounded Banach Mazur distance.  Therefore $X_\epsilon$ is isomorphic to the $l^1$ sum of
to $W^1_1(\mathbb{R}^2)$ spaces with Banach-Mazur distance independent of $\epsilon$. Hence, the space $X_\epsilon$ has BAP with constant independent of $\epsilon$.

By II.E.12b in \cite{Woj} there exists finite rank operator $H_\ve: X_\ve\to X_\ve$ such that $H_\ve h=h$ for $h\in J_\ve\circ R_\ve\circ(Id-T_\ve) E$ and $\|H_\ve\|\leq C(\eta)$, where $C(\eta)$ depend only from the BAP constant of $X_\ve$ and, thus, not from $\ve$.

We have
$$
H_\ve\circ J_\ve\circ R_\ve\circ(Id-T\ve)f=J_\ve\circ R_\ve\circ(Id-T_\ve)f\mbox{ for $f\in E$},
$$
It follows that the composition $J^{-1}_\ve\circ H_\ve\circ J_\ve\circ R_\ve\circ(Id-T_\ve): W^1_1(\Omega)\to Y_\ve$, which is a finite rank operator, satisfies for $f\in E$,
$$
J^{-1}_\ve\circ H_\ve\circ J_\ve\circ R\ve\circ(Id-T_\ve)f= R_\ve\circ(Id-T_\ve)f.
\eqno{(4)}
$$
Let us define now $S_\ve: W^1_1(\Omega)\to BV(\Omega)$ by the formula
$$
S_\ve= T_\ve+J_\ve^{-1}\circ H_\ve\circ J_\ve\circ R_\ve\circ(Id-T_\ve)
$$
Clearly it is a bounded finite rank operator with norm independent of $\ve$. Moreover,
if $f\in E$ then by (4),
$$
S_\ve f=T_\ve f+R_\ve\circ(Id-T_\ve)f
$$
and therefore
$$
\|f-S_\ve f\|_{BV(\Omega)}=\|(Id-R_\ve)\circ (Id-T_\ve)f\|_{BV(\Omega)}.
\eqno{(5)}
$$
Since
$$
(Id-R_\ve)\circ (Id-T_\ve)f=\sum_{B\in\bbb_\ve}f|_B - (\tra_B f)_{\Lambda(B)}\cdot h_B
$$
we get
\begin{eqnarray*}
\|f-S_\ve f\|_{BV(\Omega)}&=
\sum_{B\in\bbb_\ve}|(\tra_B f)_{\Lambda(B)}|\cdot\int_B|\nabla h_B|+|h_B|\,d\lambda_2\\&\qquad+
\sum_{B\in\bbb_\ve}\int_{\Lambda(B)}|\tra_B f-(\tra_B f)_{\Lambda(B)}\cdot\tra_B h_B|\,d\lambda_1\\
&\qquad\qquad+\sum_{B\in\bbb_\ve}\int_B|\nabla f|
+|f|\,d\lambda_2 \\
&\leq \sum_{B\in\bbb_\ve}|(\tra_B f)_{\Lambda(B)}|\cdot
\big(\|h_B\|_{BV(B)}+\int_{\Lambda(B)}|1-\tra_B h_B|\,d\lambda_1\big)\\
&\qquad+\sum_{B\in\bbb_\ve}\int_{\Lambda(B)}|\tra_B f-(\tra_B f)_{\Lambda(B)}|\,d\lambda_1\\
&\qquad\qquad+\sum_{B\in\bbb_\ve}\|f\|_{BV(B)}
\end{eqnarray*}

Applying  now Lemma \ref{L27} to the first sum we get bound
$C\cdot \sum_{B\in\bbb_\ve}\big(\|f\|_{BV(B)}+\|f\|_{BV(A_B)}\big)$. For the second sum
we get the same bound using the triangle inequality
\begin{eqnarray*}
\sum_{B\in\bbb_\ve}\int_{\Lambda(B)}|\tra_Bf-(\tra_B f)_{\Lambda(B)}|\,d\lambda_1&\leq
\sum_{B\in\bbb_\ve}\int_{\Lambda(B)}|\tra_Bf-f_{A_B}|\,d\lambda_1\\
&\qquad +\sum_{B\in\bbb_\ve}\int_{\Lambda(B)}|f_{A_B}-(\tra_B f)_{\Lambda(B)}|\,d\lambda_1
\end{eqnarray*}
and Lemma 2.4 and  2.6 respectively.

Therefore we get

\begin{eqnarray*}
\|f-S_\ve f\|_{BV(\Omega)}&\leq C\cdot \sum_{B\in\bbb_\ve}
\big(\|f\|_{BV(B)}+\|f\|_{BV(A_B)}\big)\\
&\leq\int_{\Omega\cap (\rr\setminus\Omega)_\ve}|f|+|\nabla f|\,d\lambda_2
\end{eqnarray*}
By our choice of $\ve$, for $f\in E$ we have
$$
\|f-S_\ve f\|_{BV(\Omega)}<\delta\cdot\|f\|_{W^1_1(\Omega)}.
\eqno{(6)}
$$

Operator $S_\ve$ is almost what we need to establish BAP with one exception. Instead of $W^1_1(\Omega)$
its rank is contained in $BV(\Omega)$. To remedy that we use the result of \cite{W} which says that
$BV(\Omega)$ is contained in the second dual of $W^1_1(\Omega)$ and the variant of the principle of local reflexivity. Namely we have

\begin{lemma}\label{lrl}(cf. \cite{W}, Thm 4)
There is embedding $\delta: BV(\Omega)\to W^1_1(\Omega)^{\ast\ast}$ such that $\delta\circ\iota=\chi$
where $\iota: W^1_1(\Omega)\to BV(\Omega)$ is a natural embedding and $\chi: W^1_1(\Omega)\to W^1_1(\Omega)^{\ast\ast}$ is a canonical embedding.
\end{lemma}

Put $S_\ve': W^1_1(\Omega)\to BV(\Omega)$ by $S_\ve'=S_\ve+\Pi_E\circ\iota-S_\ve\circ\Pi_E\circ\iota$ where $\Pi_E: BV(\Omega)\to E$ is any projection on $E$ of norm less then $\dim E$. Obviously $S_\ve'f=f$ for $f\in E$ and, by (6),
$$
\|S_\ve'f\|=\|S_\ve f+(Id-S_\ve)\circ\Pi_E\circ\iota f\|\leq \|S_\ve f\|+\delta\|\Pi_E\circ \iota f\|\leq (\| S_\ve\|+\delta{\rm dim}E)\|f\|,
$$
and for each $E$ we can start from choosing $\delta=1/\rm{dim} E$, so that $\|S_\ve'\|\leq C(\eta)+1$.

Applying now Lemma \ref{lrl}, we obtain a finite dimensional space $\delta \circ S_\ve'(W^1_1(\Omega))\subset W_1^1(\Omega)^{**}$ and the principle of local reflexivity yields the existence a norm $2$ operator
$U_\ve:\delta\circ S_\ve'(W^1_1(\Omega))\to W^1_1(\Omega)$ such that $U_\ve f=f$ for $f\in E$. Hence $U_\ve\circ S_\ve'$ is a bounded finite rank operator whose restriction to $E$ coincide with identity, and a bound on its norm depends only from $\eta$ and $k$. This completes the proof of Theorem \ref{main}.

\section{Decomposition of $\Omega$}\label{decomposition}
\subsection{More definitions and notation}
For two points $a,b\in{\mathbf R}^2$ we will denote the closed interval with ends in $a$ and $b$ by $[a,b]$ (we allow $[a,b]$ to be a point if $a=b$). The open interval with ends in $a$ and $b$ we denote by $(a,b)$, and the half-open intervals by $[a,b)$ and $(a,b]$. The line passing through the points $a,b\in {\mathbf R}^2$ (where $a\neq b$) will be denoted by $\overline{ab}$. In both cases whenever the direction matters we suppose it go from $a$ to $b$.

We denote an open ball with center in $x$ and radius $r$ by $D(x,r)$.
The left half-plane of the two parts into which the line $\overline{ab}$ cuts $\mathbf{R}^2$ will be denoted by $L(\overline{ab})$. We denote the closure of $E$ in the relative topology of $\Omega$ as $\overline{E}^\omega=\overline{E}\cap\Omega$.

We will often not distinguish between a parameterized curve and its image on the plane when it does not lead to a confusion. When we consider a union of two (or more) curves we assume it is parameterized in the natural way according to the parameterizations of the curves in the union (the second curve starts where the first ends).

Through this paper we use the notion of polygon slightly  more general than usual. Given a finite collection of points $V=(v_k)_1^n$ we say that $V$ are vertices of a {\it polygon} $P$, if $P$ is a connected simply connected bounded open planar domain and the closed curve $\gamma_V=\bigcup\limits_{j=1}^n[v_j,v_{j+1}]$ is the boundary of $P$ (here and further we assume $v_{n+1}:=v_1$). We would say that $\gamma_V$ is counterclockwise (or positive) oriented boundary of $P$ if for $j=1,\ldots,n$ for every point $x$ of the interval $(v_j,v_{j+1})$ there exists $r>0$ such that $D(x,r)\cup L(\overline{v_jv_{j+1}})\subset P$.
The intervals $[v_j,v_{j+1}]$ will be called {\it sides} of the polygon $P$. Notice that we do not exclude the possibility that a side has length $0$ or that two or more of consecutive sides belong to the same line. We also do not expect the curve $\gamma_V$ to be simple, see for example the Figure 1, yet we often mean in the notations a universal covering of $P\cup\partial P$ in which the boundary is simple.


In our construction we will use a special type of polygons (the example on the Figure 1 happens to be of this special type). To introduce the type of polygons we fix a coordinate system, so that we can talk about vertical/horizontal direction.

\begin{definition}We denote as $P=P(v_1,\ldots,v_{2n})$ a polygon whose vertices $(v_j)_{j=1}^{2n}$ are numbered in the positive direction. We say that $P(v_1,\ldots, v_{2n})\in VH(\alpha,\beta)$ if the following properties are satisfied:
\begin{enumerate}
\item The sides $[v_{2j+1}, v_{2j+2}]$ are vertical;
\item The sides $[v_{2j}, v_{2j+1}]$ are horizontal;
\item Some sides may have length zero, i.e. $v_k=v_{k+1}$, but always $v_{k}\neq v_{k+2}$, i.e. there are no two or more adjoint sides of zero length.
\item All sides of $P(v_1,\ldots,v_{2n})$ have length bounded by $\alpha$;
\item For any $k$ the rectangle $Q([v_k,v_{k+1}],P):=P(v_k,v_{k+1},c,d)$ such that $|[v_{k+1}, c]|=\beta$ is a subset of $P$ (see Figure 2a). Notice that the vertices are numbered in the positive direction, so the rectangle is uniquely defined even if the boundary curve is not simple.
\item If the angle $\angle (v_{k+1} v_{k} v_{k-1})=180^\circ$ then $[q,v_k)\subset P$ where $\angle (v_{k+1}v_k q)=90^\circ$ and $d(v_k,q)=\beta$.
\item If the angle $\angle (v_{k+1} v_{k} v_{k-1})$ (measured in the positive direction, i.e. interior to $P$) is $270^\circ$, then
there exists a closed square $Q(v_k,P):=P(v_{k}, c, d, e)$ with vertical or horizontal sides of length $\beta$ which consists
only  of interior points of $P$, except for $v_{k}$ (see Figure 2b). Notice that if the angle at $v_k$ is $270^\circ$ only one quadrant has sides which do not contain the points of $\partial P$ close to $v_k$, thus the square is uniquely defined.
\item If the angle $\angle(v_{k+1} v_{k} v_{k-1})$ (measured in the positive direction, i.e. interior to $P$) is $360^\circ$, then there exists two adjoint closed squares $P(v_{k}, c, d, e)$ with vertical or horizontal sides of length $\beta$ consisting
only  of interior points of $P$, except for $v_{k}$ (see Figure 2c). Those squares are unique and their union is denote as $Q(v_k,P)$.
\end{enumerate}
\end{definition}

\subsection{Overview of the construction} The decomposition will be done in the following way.

Our construction depends on a fixed parameter $d=d(\ve)>0$ which can be chosen to fit $\ve$. We begin with a construction of a polygon $P_d\in VH(3d,\frac12d)$, which is "inscribed" in the domain $\Omega$ in the following sense.

\begin{proposition}\label{polygon} Given a domain $\Omega$ and a number $d>0$ there is a polygon $P_d\subset \Omega$, $P_d\in VH(3d,\frac12 d)$ such that every side of $P_d$ (including zero-length ones) contains a point of $\partial \Omega$.
\end{proposition}

Next, we cut the polygon $P_d$ into $\eta$-elementary pieces with disjoint interiors of diameter comparable to $d$ to get a family $\mathcal{A}_\ve$, i.e. $\overline{P_d}^\omega=\bigcup\limits_{A\in \mathcal{A}_\ve}\overline{A}^\omega$ (this part of the construction is descried in details in section \ref{sec:cutpolygon}).

Finally, we divide the reminder $\Omega\setminus \overline{P_d}^\omega$ to get the family $\bbb_\ve$ so that $\aaa_\ve$ and $\bbb_\ve$ fulfil the property (\ref{dec.dec}) of the Partition Theorem. In order to attain the later the collection $\bbb_\ve$ is obtained in the following way: $\Omega$-closure of every component of $\Omega\setminus \overline{P_d}^\omega$ has a non-empty intersection (an arc of $\partial P_d$) with $\Omega$-closure of only one set from $\aaa_\ve$ (as one will see from the construction of $\aaa_\ve$). An elements of $\bbb_\ve$ are then the union of the components of $\Omega\setminus \overline{P_d}^\omega$ which has $\Omega$-closure with non-empty intersection with the $\Omega$-closure of the same element of $\aaa_\ve$ - see the Figure 3. The properties  (\ref{dec.oneB}), (\ref{dec.oneA}), and the corresponding part of (\ref{dec.disjoint}) follow immediately from the way of constructing $\bbb_\ve$.


Observe that $\bigcup\limits_{B\in \bbb_\ve}B=\Omega\setminus \overline{P_d}^\omega$. We will show that the construction implies that $\bigcup\limits_{B\in\bbb_\ve}B\subset (\rr\setminus\Omega)_\ve$. This gives us the first part of the property (\ref{dec.approx}) for the partition.

The partition $\aaa_\ve$ of the polygon $P_d$ will be based on the following proposition and scaling.

\begin{proposition}\label{cut.polygon} There exist universal constants $\Delta,\delta,\eta >0 $, such that, given a polygon $P\in VH(3,\frac12)$, and a collection\footnote{Those points in our construction are $\partial P\cap \partial \Omega$} of points $\mathcal{P}\subset \partial P$, such that for every side of $P$ the set $[v_{k},v_{k+1}]\cap \mathcal{P}$ is non-empty, we can decompose the polygon as $\overline{P}^\omega=\cup \overline{A_m}^\omega$. Here  $\aaa=\{A_m\}$ is a collection of disjoint $\eta$-elementary sets of diameter less than $\Delta$ and greater than $\delta$. For every $A_m$ the set $\partial A_m\cap\partial P$ is either an arc of $\partial P$ connecting two points of $\mathcal{P}$ or empty.
\end{proposition}

Now, the properties (\ref{dec.Aclosed}), (\ref{dec.diam}) and the remaining part of the property (\ref{dec.disjoint}) of the Partition Theorem follow immediately. The property (\ref{dec.fewA}) follows to the fact that all the sets in $\aaa$ are $\eta$-elementary and have comparable sizes. We explain this implication in the following lemma.

\begin{lemma} For any $\Delta>\delta>0$ and $\eta>0$ there exists $k=k(\eta,\delta,\Delta)$, such that if $\{A_n\}$ is an arbitrary collection of disjoint open $\eta$-elementary sets for which $\delta<diam(A_n)<\Delta$, then the closure of any set in $\{A_n\}$ intersects the closures of no more than $k$ sets from the collection.
\end{lemma}

\begin{proof} There exists a constant $c=c(\eta)$ such that area of an $\eta$-elementary set $A$ is at least $c \cdot (diam(A))^2$. Consider a point $p\in A_n$. Then every set $A_m$, such that $\overline{A_n}\cap\overline{A_m}\neq \emptyset$, is contained in a disc with centra in $p$ and radius $2\Delta$. At the same time the sets $\{A_n\}$ are disjoint and each has the area at least $c\delta^2$. This means that the number of sets neighboring $A_n$ does not exceed $k=\frac{4\pi\Delta^2}{c\delta^2}$.
\end{proof}

It remains to observe that, as diameters of the sets of $\aaa_\ve$ are at most $\Delta d$ and $\partial A_B\cap \partial P_d\neq \emptyset$ for $B\in\bbb_\ve$, we have $\bigcup\limits_{B\in \bbb_\ve}A_B\subset (\partial P_d+D(0,\Delta d))\subset (\mathcal{P}+D(0,(\Delta+3)d)\subset (\partial\Omega+D(0,(\Delta+3)d)\subset (\rr\setminus \Omega)_\ve$ when $d<\ve/(\Delta+3)$, which gives the remaining part of the property (viii).

Thus, Theorem \ref{partition} will be established as long as we prove
Propositions \ref{polygon} and \ref{cut.polygon}.
Their proofs constitute the rest of the paper.

\subsection{Proof of the proposition \ref{polygon}}

We begin by constructing a (finite) sequence of increasing by inclusion $VH(d,\frac12 d)$-polygons $(P_k^d)_{k=1}^N$, $P_k^d\subset\Omega$. Then we will build a $VH(3d,\frac12d)$-polygon $P_d$ with the required properties such that $P_N^d\subset P_d\subset \Omega$.

Fix $x\in \Omega$. Consider the lattice of  squares
(in a fixed coordinate system) with the side-length $d$. We denote by $\Omega_d$ the union of all closed lattice squares contained in $\Omega$. Assuming that the given $d$ is small enough (less than $\frac1{\sqrt{2}}dist(x,\partial\Omega)$),
we have $x\in (\Omega_d)^\circ$.  The polygon
$P_0^d=\Omega_d(x)$ is the connected component of
$(\Omega_d)^\circ$ which contains the point $x$. It is easy to see that
$P_d\supset\Omega_d(x)$ and $\cup \Omega_{2^{-k}}(x)=\Omega$. Thus, $\cup P_{2^{-k}}$ is an open covering of $\Omega$ with $P_{2^{-k}}\subset P_{2^{-(k+1)}}$, so that for $k$ large enough $\Omega\setminus P_{2^{-k}}\subset (\rr\setminus\Omega)_\ve$, which gives us the first part of the property (viii).



The rest of the construction can be re-scaled, so we may assume from now on that $d=1$ and omit this index in our notations.

Directly from the definition it follows that $P_0$ is a
connected polygon with the sides parallel to the axis, whose vertices have integer coordinates.
It is a straightforward observation that $P_0$ has to be simply connected and its boundary is a simple curve (see Figure 4).

Adding if necessary some zero length sides at the
integer-coordinates points of $\partial P_0$, we can assume that $P_0\in VH(1,1)$ and the length of each side is $1$ or zero.


The idea of the construction of the sequence $(P_n)$ is that we increase the polygon by "moving its sides out" as long as possible (see Figure 4). This will be done in a sequence of 3 step-procedures, which slightly differs when we build $P_1$ from $P_0$, and when we build further $P_{k+1}$ from $P_k$, $k\geq 1$. To make the first step we need the following.


\begin{lemma}\label{rectangleonaside} If $P=P(v_1,\ldots, v_{2n})\in VH(\alpha,\beta)$ and $P\subset\Omega$ then for
any $j$ for which $v_j\neq v_{j+1}$, there
exists unique pair of points $p,q\in\rr$ such that $R=P(v_{j+1},v_j,p,q)$ is an open rectangle, $[p,q]$ contains a point of $\partial\Omega$, and $\overline{R}\setminus [p,q]\subset\Omega$.
\end{lemma}

\begin{proof}As $P\subset\Omega$, $[v_j,v_j+1]\subset\overline{P}\subset\Omega\cup\partial\Omega$. If $[v_j,v_{j+1}]$ contains a point of $\partial \Omega$, we can take $p=v_j, q=v_{j+1}$. Therefore  we may assume that $[v_j,v_{j+1}]\subset \Omega$.

Consider the union $R=\cup P(v_{j+1},v_j,\xi,\zeta)$ of all open
rectangles with a side $[v_{j+1}, v_j]$ (in that direction), such that $ \overline{P(v_{j+1},v_j,\xi,\zeta)}\subset \Omega$. Since we keep  fixed the orientation all the rectangles are on the
"outer" side of the polygon $P$ (but they may intersect $P$ if the length of the side
$[v_j, \xi]$ is big enough).

As $\Omega$ is bounded, so is $R$, as $R\subset \Omega$. As an increasing union of open rectangles with a common side, the union $R$ is an open rectangle, and we may say $R=P(v_{j+1},v_j,p,q)$ (if $R=\emptyset$ then $p=v_j, q=v_{j+1}$).

On one hand, $\overline{R}\setminus[p,q]\subset \cup\overline{P(v_{j+1},v_j,\xi,\zeta)}\subset \Omega$ where the union is taken over the same collection of rectangles which defines $R$.

On the other hand, if $[p,q]\cap\partial\Omega=\emptyset$, then the distance from $[p,q]$ to $\partial \Omega$ is strictly positive, and thus there exists $\xi',\zeta'$, such that $P(v_{j+1},v_j,\xi',\zeta')\supsetneq R$ and $\overline{P(v_{j+1},v_j,\xi',\zeta')}\subset \Omega$, which contradicts the definition of $R$.
\end{proof}

Notice that the rectangle given by the Lemma \ref{rectangleonaside} is unique.

{\bf Notation.} The side $[v_{j},v_{j+1}]$ of the rectangle $R$ provided by the Lemma \ref{rectangleonaside} will be called its {\it base} and the opposite side $[p,q]$ - its {\it top}.

\subsubsection{Construction of $P_1$ starting from $P_0$.}

Clearly all sides of $P_0$ are contained in $\Omega$.

We first construct $\tilde{P_0}$ by adding to $P_0$ all rectangles provided by Lemma \ref{rectangleonaside} which are build on its non-zero horizontal sides together with the corresponding  bases without(!) the bases' endpoints. This means that we increase amount of the vertices of $P_0$ (see Figure 5).

We claim that $\tilde{P}_0$ is a polygon, i.e. connected and simply connected, and that $\tilde{P_0}\in VH(1,1)$.

To show this we observe that none of the added rectangles intersect neither $P_0$ nor any other
added rectangle. Indeed, if the height of any added rectangle $R$ would exceed $1$, then $R$ contains a lattice square $Q$ with the same base and such that $\overline{Q}\subset\Omega$. But then $Q\subset \Omega_d(x)$, and its base can not be a side of $P_0$. But $R$ can not intersect $P_0$ (which consists of the lattice squares) unless its height is greater than $1$.

On the other hand, a pair of added rectangles $R_1$ and $R_2$ can intersect each other only if their bases are the opposite (horizontal) sides of a lattice square, and each of them contains the top of the other (see Figure 6). But the top of the rectangle $R_1$ contains points of $\partial \Omega$, and thus can not be part of $\overline{R_2}$, unless it coincides with its top. But, if the closures of the rectangles intersect only at the tops, the rectangles themselves, being open, do not intersect each other.


When we add a rectangle and the side to which it is attached to the polygon, the union is obviously connected. As the union is homotopic to the initial polygon, it remains simply connected. As the sides of $\tilde{P}_0$ are either the sides of $P_0$ or the sides of the added rectangles, non of the sides of $\tilde{P}_0$ is longer than $1$. It is also plain to see that adding a rectangle of the width $1$ to a horizontal side of $VH(.,1)$-polygon gives us again a $VH(.,1)$-polygon. Thus, the resulting polygon $\tilde{P}_0$ is a $VH(1,1)$-polygon with additional property that each of its non-zero horizontal side contains a point from $\partial\Omega$.


In the next step we consider the zero-length horizontal sides of $\tilde{P}_0$. They can be of two types: those with corresponding interior angle of the polygon $\tilde{P}_0$ equal $180^\circ$ (we let them be), and those with corresponding interior angle of $\tilde{P}_0$ equal $360^\circ$. We deal with the later case at this (second) step of our construction.

If the angle at a zero length horizontal side $[v_j,v_{j+1}]$ is $360^\circ$, this means that one of the adjoint sides $[v_{j-1},v_j]$ and $[v_{j+1},v_{j+2}]$ is contained in the other one. W.l.o.g. we assume $[v_{j-1},v_j]\subset [v_{j+1},v_{j+2}]$.

There are two possibilities:
\begin{enumerate}
  \item $\partial\Omega\cap (v_{j-1},v_j]=\emptyset$; in this case we remove $v_j=v_{j+1}$ from the sequence of the vertices of the polygon, thus adding to the polygon the half-open interval $(v_{j-1},v_j]$;
  \item $\partial\Omega\cap (v_{j-1},v_j]\neq\emptyset$; in this case we choose $p\in \partial\Omega\cap (v_{j-1},v_j]$ which is the most close to $v_j$ and replace both vertices $v_j=v_{j+1}$ by $p$, thus adding to the polygon the half-open interval $(p,v_j]$.
\end{enumerate}

It is easy to see that non of the two operations changes the connectedness or homotopy class of the polygon. Furthermore, if a polygon is a $VH(1,1)$-polygon it clearly remains a $VH(1,1)$ polygon after any of this two operation.

We perform the operation described above on every zero-length horizontal side of $\tilde{P}_0$ with the interior angle $360^\circ$ which doesn't contain a point of $\partial\Omega$. In the end of this step we obtain a polygon which we call $P^*_1$.

Our aim is to construct a sequence of polygons $(P_k)_{k=1}^N$, such that $P_k\in VH(1,\frac12)$, and for every odd $k$ each horizontal side should contain a point of $\partial \Omega$, except for those of zero-length at which the interior angle of $P_k$ is $180^\circ$ (the later may or may not contain points of $\partial\Omega$). Also every vertical side of length less than $\frac12$ should be adjoint to an angle of at least $270^\circ$ (we call this additional conditions $*$-condition). For an even $k$ the same additional conditions should be satisfied with horizontal replaces by vertical and vice versa.It is the second part of the $*$-condition that $P^*_1$ may not satisfy.

To obtain the polygon $P_1$ it remains to take care of the short vertical sides. As $P^*_1\in VH(1,1)\subset VH(1,\frac12)$, if a side $[v_j, v_{j+1}]$ is shorter than $\frac12$ the angles at $v_j$ and $v_{j+1}$ can not both be $90^\circ$. If the one of them is $270^\circ$ or greater, then the side poses no problem. Consider the last remaining case when one of the angles is $180^\circ$. W.l.o.g. we may assume that it is the angle at $v_j$ which is $180^\circ$. This means that the side is adjoint to a horizontal side of zero-length, i.e. $v_j=v_{j-1}$. In this case we remove both vertices $v_j$ and $v_{j-1}$ from the sequence of the vertices so that the side $[v_j,v_{j+1}]$ merges with the previous vertical side. Repeating this operation as many times as needed, we get a $VH(.,\frac12)$ - polygon which satisfies the rest of the conditions, but may have some vertical sides of length greater than $1$. We remedy this last problem by adding some zero-length horizontal sides to subdivide the longer sides in pieces of length greater than $\frac12$ but not greater than $1$, and thus obtain the polygon $P_1$.

\subsubsection{Construction of $P_{k+1}$}

Assume now that we have constructed $P_k\in VH(1,\frac12)$, which  satisfies $*$-condition. W.l.o.g. we may assume that $k$ is odd, so that every horizontal side of $P_k$ (except for some zero-length with the adjoint angle of $180^\circ$) contains a point of $\partial\Omega$, and every vertical side of length less than $\frac12$ is adjoint to at least one angle of $270^\circ$ or greater.

Similarly to the construction of $P_0$ we first build $\tilde{P}_k$ by attaching to every non-zero vertical side the rectangle provided by the Lemma \ref{rectangleonaside}. Notice, that one of the sides of the added rectangle (its base) serves to merge the rectangle to $P_k$, while the remaining three become sides of $\tilde{P}_k$.

As before, we claim that none of the added (open) rectangles intersects $P_k$ nor any other added rectangle

However, the argument used to prove the disjointness in the case of $P_0$ does not work now.
Instead we argue as follows.


Suppose that a non-empty open rectangle $R_1$ attached to the vertical side $(v_j,v_{j+1})$ intersects $P_k$ (see Figure 7).

Consider first the case when $dist((v_j,v_{j+1}),R_1\cap P_k)>0$. $R_1\cap \partial P_k$ is nonempty as $R_1$ contains both points of $\Omega\setminus P_k $ and $P_k$). Let $q$ is a point of $R_1\cap \partial P_k$ which is one of the most close to $(v_j,v_{j+1})$. Such a point exists as $\partial P_k$ consists of finitely many horizontal and vertical line segments. The point $q$ belongs to a vertical side of $P_k$ (a non-zero one, as otherwise the angle at $q$ would be $0^\circ$). As $R_1$ is open we may choose $q$ to be not a vertex, and thus being an internal point of some vertical side. Let $p$ is an arbitrary point of $(v_j,v_{j+1})$.

As there are no points in $R_1\cap \overline{P_k}$ which are more close to $(v_j,v_{j+1})$ than $q$, the segment $(p,q)$ is entirely contained in $\Omega\setminus \overline{P_k}$. We may consider the curve $(p,q)=\gamma(p,q)=\gamma$. On the other hand, as the open polygon $P_k$ is connected there exists a simple curve $\gamma^*=\gamma^*(p,q)$ which connects $p$ and $q$ inside $P_k$. As $\gamma\cap\gamma^*=\{p,q\}$ the curve $\gamma\cup\gamma^*$ is a closed simple curve.

By Jordan's theorem the closed simple curve $\gamma\cup\gamma^*$ divides $\rr$ in two open domains. One of them is bounded and we denote it $I(\gamma,\gamma^*)$. As $\gamma\cup\gamma^*=\partial I(\gamma, \gamma^*)\subset \Omega$ and $\Omega$ is a simply connected domain, we know $I(\gamma,\gamma^*)\subset \Omega$.

As $\gamma\setminus\{p,q\}\subset \Omega\setminus \overline{P_k}$ and $\gamma^*\setminus\{p,q\}\subset P_k$ we have $\partial P_k\cap I(\gamma,\gamma^*)\neq \emptyset$. But $\partial P_k$ is a curve consisting of vertical and horizontal sides of $P_k$. As $\partial P_k\cap \partial I(\gamma,\gamma^*)=\{p,q\}$ where $p$ and $q$ are internal points of vertical sides, every horizontal side of $P_k$ which intersects $I(\gamma,\gamma^*)$ is contained in it. This means that $\partial P_k\cap I(\gamma,\gamma^*)$ is a curve which contains a non-zero horizontal side of $P_k$. But then by the $*$-condition the horizontal side contains a point of $\partial \Omega$, i.e. $\partial \Omega\cup I(\gamma,\gamma^*)\neq \emptyset$. This is a contradiction to $I(\gamma,\gamma^*)\subset\Omega$.

In the case when $d((v_j,v_{j+1}),R_1\cap P_k)=0$, the rectangle $P_k$ having only finitely many sides, this means that there exists another side of $P_k$, say $(v_m,v_{m+1})$, such that $(v_j,v_{j+1})\cap (v_m,v_{m+1})$ is non-empty. If we can take $p\in (v_j,v_{j+1})\cap (v_m,v_{m+1})\cap \Omega$, then we may choose a simple closed curve $\gamma\subset P_k\cup \{p\}$, which contains a small horizontal segment with centrum in $p$. Again, by Jordan's theorem we can consider a bounded open domain $I(\gamma)$, $\partial I(\gamma)=\gamma$, $I(\gamma)\subset \Omega$. As the domain contains part of $(v_j,v_{j+1})$ it contains part of the (not necessary simple) closed curve $\partial P_k$. Assuming the natural parametrization $\phi$ of $\partial P_k$ by $\mathbb{T}$ we see that $\phi^{-1}(p)$ consists of two points which divide $\mathbb{T}$ in two arcs of which one (say $\tau$) parameterizes $I(\gamma)\cap \partial P_k$. It is easy to see that $\phi(\tau)$ contains at least one horizontal side which is either non-zero, or is adjoint to the interior angle of $360^\circ$. Again, by the $*$-condition, this implies $I(\gamma)\cap \partial \Omega\neq \emptyset$ which contradiction proves that $(v_j,v_{j+1})\cap (v_m,v_{m+1})\subset \partial \Omega$. But then the rectangle added to the side $(v_j,v_{j+1})$ should be empty (and so does not have non empty intersection with $P_k$).

Thus, we have shown that non of the added at this step rectangles can intersect $P_k$.

Let, now a rectangle $R_1$ added to a side $(v_j,v_{j+1})$ intersects a rectangle $R_2$ added to a side $(v_m,v_{m+1})$. Choose $p\in (v_j,v_{j+1}), q\in (v_m,v_{m+1})$ and $r\in R_1\cap R_2$. As we have shown above $R_1,R_2\in \Omega\setminus P_k$, so the curve $\gamma=\gamma(p,q)=[p,r]\cup [r,q]$ satisfies $\gamma\cap (\Omega\setminus \overline{P_k})=\{p,q\}$. We can choose $\gamma^*=\gamma^*(p,q)$ to be a simple curve connecting $p$ and $q$ inside $P_k$ and the further arguments works exactly as the one in the first case above.

We have, thus, proven that non of the added to $P_k$ rectangles intersects $P_k$ or each other.

Connectedness and simply connectedness of the resulting polygon $\tilde{P}_k$ follows in the same way as in the case of $P_0$.

Now we verify that $\tilde{P}_k\in VH(.,\frac12)$. Indeed, if we have added a rectangle to a side of length at least $\frac12$, the property (v) of being $VH(.,\frac12)$ is preserved automatically. When we add a rectangle to a side of length less than $\frac12$ one of the adjoint angles has to be at least $270^\circ$ (and as the side is non-zero it has to be exactly $270^\circ$). In this case one of the horizontal sides of the added rectangle is a subset of the adjoint horizontal side of the polygon $P_k$ (the side of the rectangle can not be longer than the adjoint side of $P_k$, as the later contain a point of $\partial \Omega$ while the former does not). So, modifying our previous strategy, we add the half-open horizontal side of the rectangle which is a part of the horizontal side of $P_k$ along with the rectangle (see Figure 8). This modification of the construction does not change the connectedness or simply connectedness of $\tilde{P}_k$ but provides $\tilde{P}_k$ with $VH(.,\frac12)$ property.

When $\tilde{P}_k$ is constructed we proceed in exactly the same manner as in the construction of $P_1$ to obtain $P^*_{k+1}$ and finally $P_{k+1}$.

We proceed to construct the sequence of the polygons $P_k$, till the first
moment at which the area of every
added rectangle is less than $\frac14$. Such step will eventually
come, since otherwise the polygons $P_k$ have areas
increasing at each step by at least $\frac14$ still being subsets of $\Omega$, which would contradict the boundedness of $\Omega$.

\subsubsection{Construction of the polygon $P$}

Let us consider $P_n$, the polygon we obtained in the step on which the iterative process
described above stopped. We already know that $P_n\in VH(1,\frac12)$ is connected and simply connected polygon.
W.l.o.g. we may assume that the construction stopped at an odd step, so we know that each horizontal side (apart from some of zero length which are adjoined to an angle of $180^\circ$) contains a point of $\partial\Omega$. This means that we luck points of $\partial \Omega$ only on vertical sides, or zero-horizontal sides adjoined to an angle of $180^\circ$.

Consider the rectangles added in the last step of our construction. Non of them has area greater than $\frac14$, which means that all the rectangles added to the horizontal side of length at least $\frac12$ has the vertical sides no longer than $\frac12$. At the same time the rectangles added to the sides of length less than $\frac12$ are no longer than the adjoint by an angle of $270^\circ$ vertical side of $P_{n-1}$ due to the $*$-condition, i.e. are not longer than $1$.

Let $[v_k,v_l]=\bigcup_{j=k}^{l-1}[v_j,v_{j+1}]$ be the union of any maximal chain of consecutive sides of $P_n$ contained in a single vertical segment. This means that
$[v_{k-1},v_k]$ and $[v_l,v_{l+1}]$ are horizontal sides of $P_n$ of non-zero length and all horizontal sides in the union are of zero length. All vertical side in the union,  except possibly $[v_{k},v_{k+1}]$ and $[v_{l-1},v_l]$ which could be vertical sides of the added on the last step rectangles, contain points of $\partial\Omega$, say $q_j\in[v_j,v_{j+1}]\cap \partial\Omega$ for $j\equiv k \mbox{ mod } 2$, $k<j<l-1$. Then we transform the collection of sides of the polygon $P_n$ replacing its sides $[v_k,v_{k+1}],\dots,[v_{l-1},v_l]$
by the new sides
$$
[v_k,q_{k+2}],\quad[q_{k+2},q_{k+2}],\quad [q_{k+2},q_{k+4}], \quad\dots,
\quad[q_{l-5},q_{l-3}],\quad [q_{l-3},q_{l-3}],\quad [q_{l-3},v_l]
$$
(if $l=k+3$ then we make just one new side $[v_k,v_{l}]$).

Note that the union of
the sides remains unchanged, but every new side contains some point of $\partial\Omega$. Moreover,
the new sides has length not exceeding 2. It may happen however that some of the new vertical
sides may be shorter then $\frac12$. Then we make yet another transformation of the polygon
aggregating consecutive short sides together or to the new longer ones. The resulting combination of sides consists of the sides not longer then 3.


Proceeding in that way with all vertical sides we finally get a polygon $P_n'\in VH(3,\frac12)$ whose every horizontal and vertical side contain a point of $\partial\Omega$, except for some of those which are the vertical sides of the added on the last step rectangles (not connected to a non-zero vertical side of $P_{k-1}$).

Those later could arise only as a difference between two sides of the added on the last step rectangles, which means they are necessarily adjoined to an angle of $270^*$ (or the vertical side is a zero-length side adjoined to an angle of $180^\circ$). For those non-zero sides only we make one last addition of the rectangles provided by the Lemma \ref{rectangleonaside} along with the corresponding base and one of the horizontal sides. If it is a zero-length side we shift it till is meets a point of $\partial\Omega$. As every horizontal side contains a point of $\partial \Omega$ the vertical sides won't shift at a distance greater than $1$. Thus the length of the modified horizontal sides is not greater than $3$.

We see that the polygon obtained as the result of this last adjustment is the desired in the Proposition \ref{polygon} polynomial $P$: Every side now contains a point of $\partial\Omega$, and the lengthes of the sides are no more than $3$, thus $P\in VH(3,\frac12)$.


\subsection{Proof of the proposition \ref{cut.polygon}}\label{sec:cutpolygon}

Let  $P\in HV(3,\frac12)$ and $\mathcal{P}$ be a closed subset of $\partial P$ with nonempty
intersection with each side of $P$.

We will show that, there exist some universal constants $\delta,\Delta, \eta$, such that the
polygon $P$ can be cut into $\eta$-elementary
pieces $(A_j)$ with disjoint interiors, such that $\delta<diam(A_j)< \Delta$ for every $j$ and
such that every point of $\partial P\setminus \mathcal{P}$ belongs to the closure of exactly one of $A_j$'s.
In particular, every open arc of $\partial P$ disjoint with $\mathcal{P}$ intersects the closure of only one $A_j$.

As a first step, let us consider $P_\epsilon$, the intersection of $P$ and
the $\epsilon$-neighborhood of $\partial P$ where $\epsilon<\frac1{10}$.
Introduce $S_\epsilon=\partial P_\epsilon \cap P$. First we will show
that for all $\epsilon\in(0,\frac1{10})$ the set $S_\epsilon$ is a
simple rectifiable curve. Then we will show that $S_\epsilon$ is
$(10,\frac{\epsilon}2 )$-Lipschitz, i.e. that shortest sub-curve of
$S_\epsilon$ connecting two points at  Euclidian distance smaller then
$\epsilon/2$ is $10$-Lipschitz.


Then for $\epsilon=\frac1{16}$ and $\eta$ sufficiently large we will cut $P_\epsilon$ in
$\eta$-elementary pieces of diameter at least $\frac1{16}$ and at most $6$ in such a way that each point of $\partial P\setminus \mathcal{P}$ belongs to exactly one piece. Finally we will cut
$P\setminus P_\epsilon$ in $\eta$-elementary pieces of diameter at least $\frac1{64}$ and at most $\frac14$. Then our task
will be accomplished.

\subsubsection{Curve $S_\epsilon$}

We begin from the observation that with natural parametrization the boundary $\partial P$ is a
closed piecewise linear curve (not necessarily simple).
In the study of $S_\epsilon$ it will be convenient for us to remove
all zero-length sides $\tau$ adjoined to an angle $180^\circ$, so we assume that any angle of $P$ is
$90^\circ, 270^\circ$ or $360^\circ$. Still we may say that
$P\in VH(.,\frac12)$, even though the estimate on the lengthes of the sides is destroyed by the above removal.

\begin{lemma}\label{Senoisolated} $S_\epsilon$ contains no isolated points.
\end{lemma}
\begin{proof}
Observe first that $S_\epsilon$ consists of the
interior points $z\in P$ with $d(z,\partial P)=\epsilon$.

Let $y$ be an isolated point of $S_\epsilon$. Consider a circle $\mathcal{C}$ with
center  $y$ and radius $r$ such that neither the circle nor its
interior $\mathcal D$ contain other points of $S_\epsilon$ ($r$ is much less than $\epsilon$).
If there are $a,b\in\mathcal{C}\cup\mathcal{D}$ such that $d(a,\partial P)>\epsilon$ and $d(b,\partial P)<\epsilon$ then by the continuity of the distance there is $c\in\mathcal{C}\cup\mathcal{D}\setminus \{y\}$ such that $d(c,\partial P)=\epsilon$. Then $c\in S_\epsilon$ which contradicts the choice of $\mathcal{C}$.

Therefore, the points of the closed disc should have the distance to $\partial P$ either all
greater or all less than $\epsilon$. The former is impossible because $d(y,\partial P)=\epsilon$.


Therefore $d(c,\partial P)< \epsilon$ for all $c\in\mathcal{C}$

Since $d(y,\partial P)=\epsilon$ and $\partial
P$ is compact, there exists $p\in \partial P$ such that
$d(y,p)=\epsilon$. Let $p$ belongs to a side $\tau$. The distance from a point $y$ to a
line segment $\tau$ is either the
distance to the orthogonal projection of $y$ on $\tau$ (with $p$
being the image of $y$), or the distance to one of the the ends of the segment $\tau$
(in this case $\tau$ and $[y,p]$ form an angle
greater than $90^\circ$). Let us first consider the later case.

Since $[y,p]$ is a most short way between $y$ and $\partial
P$, we see that $(yp)\subset P$ and the internal angel at $p$ which contains $y$ is at least $270^\circ$ (see Figure 9-ab).
Moreover, $y\in Q(p,P)$.

Let $Q=Q(p,P)$ if the angle at $p$ is $270^\circ$, or the half of $Q(p,P)$ containing $y$ if the angle is $360^\circ$ (see Figure 9b). The side-length of $Q$ is $\frac12$.
Since $\epsilon<\frac14$,  circle $\mathcal{C}'$ with center in $y$ and radius
$\epsilon$ intersects $\partial Q$ in three points - $p$ and, say $a$ and $b$. As the angle $\angle apb$ inscribed in $\mathcal{C}'$ equals $90^\circ$, the
segment $[a,b]$ is a diameter of the circle, and $y$ is
its middle point. Let $z$ be the point of intersection of $\mathcal{C}$ and the
perpendicular bisector of $[a,b]$ outside $\triangle apb$ (remember that
$d(y,z)$ is small).
Let $\ell$ be the perpendicular bisector of $[y,z]$ (see Figure 10).

On one hand, if $q\in \partial P$ belongs to the component of $\Bbb R^2\setminus \ell$ which contains $y$, or to $\ell$ then $d(q,z)\geq d(q,y)\geq \epsilon$, on the other hand if $q\notin Q$ and $q$ belongs to the component of $\Bbb R^2\setminus \ell$ containing $z$ then
$d(q,z)>\epsilon=d(a,y)=d(b,y)$. Since $\partial P\cap Q=\emptyset$ it follows then $d(z,\partial P)\geq\epsilon$. This contradicts to the assumption that $d(z,\partial P)<\epsilon$.

It remains to consider the case when $p\in\partial P$
with $d(p,y)=\epsilon$ is the image of the orthogonal
projections of $y$ onto the side $\tau$ of $P$. If $p$ is one of the ends of the
side, then the angle at $p$ is at least $180^\circ$. Otherwise $p$ is
an interior point of some side $\tau$. In both cases $P\in VH(.,\frac12)$
implies that there exists a rectangle $R\subset P$, such that $p$ is an interior point to the shorter side of $R$, and the longer side of $R$ is parallel to $(y,p)$ and has length $\frac12$ (see Figure 11-ab).

Consider a sequence $\{y_j\}\subset \mathcal{D}\cap (\overline{yp}\setminus [y,p])$ which converges to $y$. By the choice of $\mathcal{D}$ it should be $d(y_j,\partial P)<\epsilon$, thus there exists a sequence $\{p_j\}\subset \partial P$, such that $d(y_j,p_j)<\epsilon$. Notice, that if $p_j\not\in (\overline{yp})_\sigma$, where $\sigma=d(y,\partial R)$. Indeed, $p_j\not\in R\subset P$, and if $x\in (\overline{yp})_\sigma\setminus R$, then $d(y_j,x)>\epsilon$.

Let $p'$ is a limit point of $\{p_j\}$. $p'\in \partial P$ and $d(y,p)=\epsilon$. If $p'$ is not the image of an orthogonal projection of $y$ to a side of $P$ we can repeat the argument of the first case. If $p'$ is the image of an orthogonal projection of $y$ to a side of $P$, then $\overline{yp'}$ is either vertical or horizontal. As $p'\not \in (\overline{yp})_\sigma$, it must be $\overline{yp}\bot \overline{yp'}$. As before, there exists a rectangle $R'\subset P$ in which the longer side is parallel to $\overline{yp'}$ and has length $\frac12$ (see Figure 12).

Consider now a sequence $\{y_j'\}\subset \mathcal{D}\cap R\cap (\overline{yp'}\setminus [y,p'])$ which converges to $y$. By the choice of $\mathcal{D}$ there exists $\{p_j'\}\subset \partial P$, such that $d(y_j',p_j')<\epsilon$. But, with the same argument as above $p_j'\not \in (\overline{yp})_\sigma\cup(\overline{yp'})_{\sigma'}$, where $\sigma'=d(y,\partial R')$. Thus, a limit point of $\{p_j'\}$ (which exists due to compactness of $\partial P$) can not be the image of orthogonal projection of $y$ on a side of $P$ and we can use the argument of the very first case of the proof.
\end{proof}

\begin{lemma}\label{structure:arcline} $S_\epsilon$ consists of finitely many arcs of
circles of radius $\epsilon$ and segments of vertical or horizontal
lines.
\end{lemma}

\begin{proof} Denote the open $\epsilon$-neighborhood of the side
$\tau_k$ by $N_\epsilon(\tau_k)$ and its boundary by
$S_\epsilon(\tau_k)$.
We have  $S_\epsilon=P\cap (\cup
(S_\epsilon(\tau)\setminus(\bigcup\limits_{\tau'\neq \tau}
N_\epsilon(\tau'))))$, where the unions run over all sides of $P$.

The set $S_\epsilon(\tau)$ consists of two
intervals parallel to $\tau_k$ and two semicircles of radius $\epsilon$.
The intersection of
$N_\epsilon(\tau)$ with any circle of radius $\epsilon$ (not centered at an interior point of $\tau$) or an interval is either an (open) arc
of the circle or an open segment of the interval.
Thus $ S_\epsilon(\tau)\setminus(\bigcup\limits_{\tau'\neq \tau}
N_\epsilon(\tau'))$ is a finite union of closed arcs, intervals or points.

By Lemma \ref{Senoisolated} $S_\epsilon$ contains no isolated points, so it
is a union of finitely many closed arcs or segments.
\end{proof}

It is convenient further to think not of the polygon $P$ in the plane, but of a covering of $\overline{P}$ by a manifold, which is isometric inside $P$ but do not glue together the points of the curve which bounds $P$, so that $\partial P$ becomes a simple curve. As $S_\epsilon$ is inside $P$ this identification does not affect its topological properties.

Let $\phi(t):[0,2n]\rightarrow \partial P$ be the natural parametrization of $\partial P$, such that each interval
$[k,k+1]$ parameterize one side of $P$. We will consider another parametrization of the curve
$$\widetilde{\phi}(t)=
\left\{
\begin{array}{ll}
    \phi(t-k), & \hbox{for $t\in[2k,2k+1]$;} \\
    \phi(k+1), & \hbox{for $t\in[2k+1,2k+2]$.}
\end{array}
    \right. $$
defined on the interval $[0,4n]$. (Observe that this is not a
simple parametrization even for a simple curve.)

We would like  $\widetilde{\psi}(t)$ to be a closest to
$\widetilde{\phi}(t)$ point of $S_\epsilon$ with respect to the
interior metrics $d_P$ of $P$, (i.e. $d_P(x,y)$ is the minimal length of a
curve connecting $x$ and $y$ inside the polygon $P$). Notice that in
this way the curve $\widetilde{\psi}$ is not yet well-defined, because several
points of $S_\epsilon$ may have the same distance to a point of
$\partial P$. To give a proper definition we have to learn more
about the set $S_\epsilon$ and the distance of its points to $\partial P$.

\begin{lemma}\label{lem.4eps} For any $p\in\partial P$ there exists $z\in S_\epsilon$
with $d_P(p,z)<4\epsilon$.
\end{lemma}

\begin{proof} Let $\tau$ be a side of $P$ and $p\in\tau$.

Consider first the case $|\tau|>2\epsilon$.
Let $\tau_\epsilon=Q(\tau,P)\cap S_\epsilon(\tau)$. It is clear that $\tau_\epsilon$ is an interval parallel to $\tau$ and of the same length contained in $Q(\tau,P)\subset P$. We have $S_\epsilon\supset\tau_\epsilon\setminus\{z\in\tau_\epsilon\,:\,d(z,\partial P)<\epsilon\}$.
Since $\partial P\cap Q(\tau,P)=\emptyset$ it is clear that
$\tau_\epsilon'=\{z\in\tau_\epsilon\,:\, d(z,\partial Q(\tau,P))>\epsilon\}\subset S_\epsilon$. It is easy to see that
$\tau_\epsilon'$ is a subinterval of $\tau_\epsilon$ of length $|\tau|-2\epsilon$ and, in particular, is not empty. The distance from any point of $\tau$ to the subinterval $\tau_\epsilon'$ is at most $\sqrt{2}\epsilon$.

Consider now side $\tau=[a,b]$ of length not exceeding $2\epsilon<\frac12$. Since $P\in VH(.,\frac12)$ one of the adjoint angles has to be greater than $90^\circ$ and, as the angles of the polygon $P$ are $90^\circ$, $270^\circ$ or $360^\circ$, is at least $270^\circ$. W.l.o.g. we can assume that it is the angle at the vertex $b$ which is $270^\circ$ or greater.

Consider a point
$x\in Q(b,P)$, such that $\overline{xb}$ makes an angle $45^\circ$ to the vertical and the horizontal direction and
$d(x,b)=\sqrt{2}\epsilon$. Since $d(x,\partial P)\geq d(x,\partial Q(b,P))=\epsilon$ and the distance to
$\partial P$ is a continuous function, there exists  $z\in[b,x]$  such that $d(z,\partial P)=\epsilon$ and $d(y,\partial P)<\epsilon$ for $y\in(b,z)$. Thus, $z\in S_\epsilon$ and for any $p\in [a,b]$ holds
$d(z,p)\leq d(z,b)+d(b,p)\leq \sqrt{2}\epsilon+2\epsilon<4\epsilon$.
\end{proof}

The Lemma \ref{lem.4eps} shows that  any point of $\partial P$ is close to $S_\epsilon$. Now we look at the most close points. For $p\in\partial P$ we denote by $m_\epsilon(p)$ the set of all points
of $S_\epsilon$ which are the closest to $p$ among the points of $S_\epsilon$.

\begin{lemma}
If $y\in m_\epsilon(q)\cap m_\epsilon(q^*)$, $y\in S_\epsilon$ for some $q\neq q^*$, then there are two sides $\tau\neq\tau'$ of $P$, such that $y\in S_\epsilon(\tau)\cap S_\epsilon(\tau')$, and the points $p\in\tau$, $p'\in \tau'$, such that $d(y,p)=d(y,p')=\epsilon$, and $p\neq p'$.
\end{lemma}

\begin{proof} Denote by $\tau$ the side of $P$ such that $y\in S_\epsilon(\tau)$ (as $S_\epsilon$ is a union of subsets of $S_\epsilon(\tau)$ for the sides of $P$, such a side exists).

Suppose $y\not \in S_\epsilon(\tau')$, for any side $\tau'\neq\tau$. Since $y\in S_\epsilon$, we know that $y$ does not belong to $\epsilon$-neighborhood of any side of $P$. And, as $y\not\in S_\epsilon(\tau')$, we know that $d(y,\tau')>\epsilon$ for every side $\tau'\neq \tau$. Thus, there exists small disc $D(y,r)$ such that $S_\epsilon\cap D(y,r)=S_\epsilon(\tau)\cap D(y,r)$, is a 1-smooth curve.

Let $p$ be the closest to $y$ point of $\tau$. Then  $d(y,p)=\epsilon$, and $(p,y)\subset P$. By the assumption, there exists another point $q^*\in\partial P$ for which $y\in m_\epsilon(q^*)$.
Let $\gamma$ be the shortest path joining $q^*$ and $y$ inside $\overline P$ (here we mean $\overline{P}$ on the universal covering). Consider the closest to $y$ point $q'\in\gamma\cap \partial P$. Then $y\in m_\epsilon(q')$ and $(q',y)\subset P$.

Suppose $p\neq q'$. Since $y\in m_\epsilon(p)\cap m_\epsilon(q')$ and $S_\epsilon\cap D(y,r)$ is a smooth curve passing through $y$, both $(y,p)$ and $(y,q')$ are orthogonal to the tangent line to $S_\epsilon\cap D(y,r)$ at $y$. Thus $p,q',y$ are collinear and,
since $(p,y)\subset P$ and  $(q',y)\subset P$, none of them can contain neither $p$ nor $q'$. This means that $y\in [p,q']$.

There are two possibilities for placement of $\overline{py}$ with respect to $\tau$:
\begin{enumerate}
\item If $p$ is an interior or endpoint of $\tau$ and $\overline{yp}\perp \tau$ so that
$(y,p)\subset Q(\tau,P)$.
\item If $p$ is an endpoint of $\tau$, then the angle at $p$ is greater than $90^\circ$ and
$(p,y)\subset Q(p,P)$
\end{enumerate}

As $q'\in\partial P$, so in the first case, $q'\not\in Q(\tau,P)$ and in the second case $q'\not \in Q(p,P)$. Then $d(p,q')\geq \frac12$ and $d(q',y)\geq \frac12 -d(y,p)=\frac12-\epsilon>4\epsilon$, which contradicts the Lemma \ref{lem.4eps}.

Thus, $p=q'\neq q^*$ and the shortest path from $q^*$ to $y$ passes through $p$. This can not happen if $p$ is an interior point of the side $\tau$.

If $p$ is an endpoint of $\tau$, then $y$ belongs to one of the semi-circles which constitute $S_\epsilon(\tau)$.
Denote the side adjoint to $\tau$ at $p$ as $\tau'$.

If the angle at $p$ is $270^\circ$, than $S_\epsilon(\tau)\setminus N_\epsilon(\tau')$ contains only half of the semicircle with center at $p$, which is also contained in $S_\epsilon(\tau')$, thus $y\in S_\epsilon(\tau')$.

If the angle at $p$ is $360^\circ$, than the whole semicircle of $S_\epsilon(\tau)$ with center at $p$ is contained in $S_\epsilon(\tau')$, and thus again $y\in S_\epsilon(\tau')$.

Assume now that $y\in S_\epsilon(\tau)\cap S_\epsilon(\tau')$ and $p\in \partial P$ is the only point for which $d(y,p)=\epsilon$. Then $p\in \tau\cap\tau'$ is the common endpoint of $\tau$ and $\tau'$ and the angle at $p$ is $270^\circ$ or $360^\circ$.

Now either there exists another side $\tau''$ such that $y\in S_\epsilon(\tau'')$ and the lemma follows or, otherwise, there exists some small disc $D(y,r)$ such that $D(y,r)\cap S_\epsilon=D(y,r)\cap S_\epsilon(\tau)$ or $D(y,r)\cap S_\epsilon=D(y,r)\cap S_\epsilon(\tau')$. When the disc exists in the later case we repeat partially the reasoning from above to get a contradiction. Indeed, suppose that there is a point $q^*\neq p$ from $\partial P$ for which
$y\in m_\epsilon(q^*)$.
Then $d_P(q^*,y)$ is realized by some curve $\gamma\in \overline{P}$, which passes through $p$.

If the angle at $p$ is $270^\circ$ then there is $q'\in \gamma\setminus [y,p]$ such that $(y,q')\subset P$ and $[q',y]$ is shorter than the part of curve $\gamma$ connecting $q'$ to $y$. This contradicts the choice of $\gamma$ as a most short path from $q^*$ to $y$.

If the angle at $p$ is $360^\circ$, consider the point $y'$ which is intersection of the line containing $\tau$ and the semicircle of $S_\epsilon(\tau)$ centered at $p$ (see Figure 13). By the properties (8) and (5) of $P\in VH(.,\frac12)$ there is an arc $D(y',r)\cap S_\epsilon(\tau)\subset S_\epsilon$. Then for a suitable chosen $y^*\in D(y',r)\cap S_\epsilon$ and $q'\in \gamma\setminus [y,p]$ holds $(y^*,q')\in P$, and the curve consisting of the part of $\gamma$ joining $q^*$ with $q'$ and $[y,q']$ is shorter than $\gamma$. But then, $d_P(y,q^*)<d_P(y,q^*)$ which contradicts to the assumption $y\in m_\epsilon(p^*)$.

\end{proof}

Now, we are ready to prove the following important fact.

\begin{lemma}\label{structur.closesttomany}If $z\in m_\epsilon(p)\cap m_\epsilon(q)$ for $p\neq q$ then there is an interval
$$[t_0,t_1]\subset [0,4n]/\{0\sim 4n\}$$ (when we talk about
intervals, we assume $0$ to be identified with $4n$), such that $m_\epsilon(\widetilde{\phi}(t))=\{z\}$
for all $t\in(t_0,t_1)$. Moreover $d(z,\widetilde{\phi}(t_0))=d(z,\widetilde{\phi}(t_1))=\epsilon$,
and $z\not\in  m_\epsilon(\widetilde{\phi}(t))$ for
$t\not\in [t_0,t_1]$.
\end{lemma}

\begin{proof} As it was shown in the previous Lemma,
there exist $p\in \tau\subset \partial P$ and $p'\in \tau'\subset
\partial P$ such that $d(z,p)=d(z,p')=\epsilon$ and $p\neq p'$, $\tau\neq \tau'$.

There are two possibilities for the position of $[z,p]$ with respect to
$\tau$ (and two for the position of $[z,p']$ with respect to $\tau'$): either $[z,p]\perp\tau$
or $p$ is an end of $\tau$ and $[z,p]$ form with $\tau$ an angle greater than $90^\circ$
(in the later case the angle at $p$ is at least $270^\circ$, see Figure 14).

Let us first consider the case when both $[z,p]\bot\tau$ and
$[z,p']\bot\tau'$ (and thus $[z,p]$ and $[z,p']$ are either horizontal or vertical).
As $P\in VH(.,\frac12)$, it is clear that
$[z,p]\bot [z,p']$.
We introduce the coordinate system with $z$ in the origin such that $p=(\epsilon,0)$ and $p'=(0,\epsilon)$ (this may change the orientation of the plane, in which case we temporary change the numbering of the vertices of the polygon $P$). We introduce $p^*=(\epsilon,\epsilon)$ so that $z,p,p^*,p'$ are the vertices of a square of the sidelength $\epsilon$ which we will call $R$ (see Figure 15). We will show next that one of the two curves into which points $p$ and $p'$ divide $\partial P$
is contained in $\overline{R}$. We denote this curve $\gamma_R$ and we will show that $z\in m_\epsilon(p)$ for $p\in\gamma_R$ and $z\notin m_\epsilon(p)$ for $p\in\partial P\setminus\gamma_R$.

Notice that $\tau \bot [z,p]$ and thus is vertical. The point $p$ may be an interior point of $\tau$ or its end. In the later case if $\tau$ is in the lower half-plain or has zero-length, the angle at $p$ is at least $270^\circ$ and $p'\in Q(p,P)\subset P$ (see Figure 16-ab) which would be impossible. So in both cases some part of $\tau$ is located in the upper half-plane, i.e. $\tau\cap \overline{\mathbb{R}\times \mathbb{R}_+}\supsetneq \{p\}$.

The point $p^*$ could not be an interior point of $\tau$
because otherwise $[z,p']\subset Q(\tau,P)$. So the end of $\tau$, which we denote by $p_1$, is either inside $(p,p^*)$ or $p_1=p^*$.

If the angle at $p_1$ is greater than $90^\circ$ then $p'\in Q(p_1,P)$. Therefore the angle between $\tau$ and the next side of $P$ (which we denote $\tau_1$) is $90^\circ$.

If $p_1\neq p^*$, denote by $p_2$ the second end of
the horizontal side $\tau_1$ (see Figure 17). The side $\tau_1=[p_1,p_2]\subset \overline R$ because all points of $(z,p')$ are
the interior points of $P$ and so $\tau_1\cap (z,p')=\emptyset$.

Hence $|\tau_1|<\epsilon<\frac12$ and as $P\in VH(.,\frac12)$ the angle at $p_2$ is not $90^\circ$.
Also the angle at $p_2$ is not $360^\circ$, as otherwise $p'\in Q(p_2,P)$.
Thus the angle at $p_2$ is $270^\circ$. Denote the
second end of the vertical side $\tau_2$ starting at $p_2$ by $p_3$ (see Figure 17).

If $p_3\not\in [p^*,p']$ we apply the previous arguments
once again. The only possibility is that the angle at $p_3$ is $90^\circ$, and at the next corner, denoted $p_4$ the angle is $270^\circ$, etc. Ultimately, after  finitely many (say $k$) such steps and not later than we reach $\tau'$, we have reached $[p^*,p']$.

This means that $p_k\in [p^*,p']$ (which includes the
case $p_1=p^*$). When it happens the side $[p_{k-1},p_k]$ is vertical, so the angle at $p_k$ is $90^\circ$. By $p_{k+1}$ we denote the next vertex of $P$. If $p_{k+1}\in (p',p*)$ then the angle at $p_{k+1}$ is not $90^\circ$ (because every side shorter than $\frac12$ is adjoint to at least one angle greater than $180^\circ$ due to $P\in VH(.\frac12)$),
and the angle at $p_{k+1}$ can not be $270^\circ$ or $360^\circ$ (because otherwise $p'\in Q(p_{k+1},P)$). Thus, $p'\in [p_k,p_{k+1}]$ and the curve $\gamma_R\subset\partial P$ connecting $p$ and $p'$ is contained in $\overline{R}$. Moreover, we have established that the angles between the sides composing the curve $\gamma_R$ alternate between $90^\circ$ and $270^\circ$.

As $P\in VH(.\frac12)$, the domain bounded by $\gamma_R\cup [z,p]\cup [z,p']$ is interior to $P$ and, thus, no segment connecting $z$ and a point of $\gamma_R$ intersects $\partial P$. This means that the  interior distance from $z$ to points of $\gamma_R$ coincide with the Euclidean one.

Now, to finish the consideration of the case when $[z,p]\bot\tau$ and
$[z,p']\bot\tau'$, it remains to show that
$m_\epsilon(r)=\{z\}$ for $r\in\gamma_R$ and
$z\notin m_\epsilon(r)$ for $r\notin\gamma_R$.

Let $r\in \gamma_R\subset \overline{R}$, $r\neq p,p'$ and
$y\in S_\epsilon$, $y\neq z$. As every point of $\overline{R}\setminus\{z,p^*\}$ is at the distance less than $\epsilon$ either from $p$ or $p'$, $y\not\in \overline{R}$.
This means that there is a curve $\gamma$ of length no more than $d(z,r)$, which connects $r$ and $y$ inside $\overline{P}$ (we consider $\overline{P}$ in the sense of the internal metric on the covering manifold of $P$).

Let $q\in\gamma$ be the last point of $\gamma$ which belongs to the domain bonded by $[z,p]\cup [z,p']\cup \gamma_R$. Clearly such a point exists and belongs to $[z,p]\cup [z,p]$. W.l.o.g. we may assume that $q\in [z,p]$. Point $y$ has to be outside of the disc $D(p,\epsilon)$. But, if $q\neq p$, since $q$ is on the radius $[p,z]$ of $D(p,\epsilon)$, the most close to $q$ point outside of the disc $D(p,\epsilon)$ is $z$. Thus $\gamma$ is strictly longer than $\gamma(r,q)\cup [q,z]$, where $\gamma(r,q)$ is the part of $\gamma$ connecting $r$ and $q$. On the other hand, if $q=p$, then
$d_P(r,y)= d_P(r,p)+d_P(p,y)\geq d_P(r,p)+d_P(p,z)>d_P(r,z)=d(r,z)$. Thus, in both cases, $d_P(r,y)>d_P(r,z)$, so $y\notin m_\epsilon(r)$, which implies $m_\epsilon(r)=\{z\}$.

We are going to show now that $z\notin m_\epsilon(q)$  for
point $q\in \partial P\setminus\gamma_R$. To this end, notice that there exist
some small disc $D(z,r)$ such that:
$$
D(z,r)\cap S_\epsilon(\tau)\setminus \overline{R}\quad\mbox{and}\quad
D(z,r)\cap S_\epsilon(\tau')\setminus \overline{R}
$$
are contained in $S_\epsilon$ (depending on
whether $p$ and $p'$ are vertices of $P$ or not those sets may be arcs and/or segments, see Figure 18).

We use the following
geometrical observation. Suppose that $\gamma$ is a curve joining $z$ with some point
$q\not\in \gamma_R$. We will show that there exists another point $z'\in D(z,r)\cap S_\epsilon$ and
a curve $\gamma'$ joining $z'$ with $q$ which is shorter then $\gamma$. Observe first, that, as $D(z,\epsilon)\subset P$, the curve $\gamma\cap D(z,\epsilon)$ is a straight segment which is not contained in the I'st or III'd quarters (the first because $q\not\in \gamma_R$, the second because $P\in VH(.,\frac12)$).

Now, consider the case when
$\gamma\cap D(z,r/2)$ contains a point $q'=(x,y)$ with one of the coordinates, say $x$, negative.
Choose then a point $z'\in (D(z,r)\cap S_\epsilon(\tau'))\setminus \overline{R}$, such that its first coordinate is $x$. Replace then the part of
$\gamma$ which joints $q'$ with $z$ by the segment $[q',z']$ (call the obtained curve $\gamma'$). $[q',z']$ is either a cathetus or a part of a cathetus in a triangle where $[q',z]$ is the hypothenuse (when the part of $S_\epsilon$ is a segment, see Figure 19-ab), or $d(z', q')<|y|<|x|<d(q',z)$ (when the part of $S_\epsilon$ is an arc, see Figure 20). In both cases the curve $\gamma'$ is shorter than the curve $\gamma$, which implies that $z$ is not a most close to $q$ point of $S_\epsilon$.

It remains to consider a case when (the one of $\gamma\supset [p',z]$ identical). The angle at $p$ can not be $90^\circ$ as $P\in VH(.,\frac12)$.

If the angle at $p$ is $270^\circ$, then $\gamma\supset [p,z]$ is possible only when $p$ is a vertex of $P$, and $\gamma$ contains also a part of the non-zero side $\tau^*$ adjoined to the side $\tau$. Consider a point $q'$, $0<d(q',p)<\epsilon$ such that $q'\in \gamma\cap \tau^*$.
Take a point $z'\in S_\epsilon(\tau)\cap D(z^*,r)\setminus \overline{R}\subset S_\epsilon$. The interval $(q',z')\subset P$ is sorter than $d(q',z)$ which shows that $\gamma$ is not a most short way from $q$ to $S_\epsilon$.

If the angle at $p$ is $360^\circ$, denote the point of intersection of the continuation of $\tau$ through $p$ with $S_\epsilon(\tau)$ by $z^*$. As in the proof of Lemma \ref{Senoisolated}, we may choose such a small radius $r$ that $D(z^*,r)\cap S_\epsilon(\tau)\subset S_\epsilon$. But then there exists a point $z'\in D(z^*,r)\cap S_\epsilon(\tau)$ with $x$-coordinate greater than $\epsilon$ (see Figure 21). Then $[q',z']\subset P$ and $d(q',p)+d(p,z)>d(q',z')$, i.e. $d_P(q,z)>d_P(q',z')$ which shows that $\gamma$ is not a most short way from $q$ to $S_\epsilon$.

Thus, we have shown that $z\not\in m_\epsilon(q)$ if $q\not\in \gamma_R$.

Summarizing the proof - so far we have shown that if
$[z,p]\bot\tau$ and $[z,p]\bot\tau'$, then $\partial P$ can be presented as a union of two curve connecting $p$ and $p'$. One of them $\gamma_R$ has the property that $m_\epsilon(r)=\{z\}$ for its interior points $r$  and $z\notin m_\epsilon(q)$ for all $q\in\partial P\setminus \gamma_R$.

Rest of the cases are rather similar to the one described, so we will omit some details. To show the change in the argument we will consider a case when $[z,p]\not\bot \tau$ and $[z,p']\not\bot \tau'$ (and leave the rest of the cases to the reader).

Recall that if $d(z,p)=\epsilon$ and $[z,p]\not\bot \tau$ then $p$ is an endpoint of $\tau$. As well, $p'$ is an endpoint of $\tau'$ and the segments
$[p,z]$, and $[p',z]$ make obtuse angles to the corresponding sides $\tau$ and $\tau'$ (and the angles at $p$ and $p'$ are at least $270^\circ$).

As $z\in Q(p,P)$ and $z\in Q(p',P)$, $Q(p,P)\cap Q(p',P)\neq \emptyset$. But then, as $d(p,p')\leq 2\epsilon$, unless $Q(p,P)$ and $Q(p',P)$ is oriented in the same way, either $p\in Q(p',P)$ or $p'\in Q(p,P)$. So, the mutual position of $z, p, p', \tau$, and $\tau'$ is the one of indicated in Figure 22 up to a symmetry or rotation. Notice that in all cases the vectors $\overrightarrow{zp}$ and $\overrightarrow{zp'}$ are in the same quadrant, which we by changing the coordinate system can assume to be the I'st quadrant.

Shift, now the origin of the coordinate system in such a way that $p$ and $p'$ are on the positive axes. Wlog we may assume $p=(x,0), p'=(0,y)$. With this placement of the coordinates $z$ is in the III quadrant. As before, we consider a rectangle $R=P((0,0),(x,0),(x,y),(0,y))$ and show that there is a sub-curve of the boundary $\partial P$, the curve $\gamma_R$ connecting $p$ and $p'$, such that $\gamma_R\subset \overline{R}$. Similarly, as above, we observe that $d_P(z,r)=d(z,r)$ for $r\in\gamma_R$.

If $d_P(r,y)\leq d_P(r,z)$ for some $r\in\gamma_R$,
we again consider a curve $\gamma$ which realizes the shortest way between $r$ and $y$ and consider $q\in\gamma\cap([z,p]\cup [z,p'])$. This leads to a contradiction in exactly the same way as before.

Some difference arises when we look at $\partial P\setminus \gamma_R$. Indeed there may even exist other point $p''$ for which $d(z,p'')=\epsilon$.

Let us shift the origin to $z$ (preserving the direction of the axes). Both $p$ and $p'$ are in the I quadrant. Consider any other point $p''\in \partial P$ at the distance $\epsilon$ from $z$. If we repeat the argument for the pair $(p,p'')$ it will lead to the same conclusion, i.e. $p''$ also is in the I quadrant (if $[z,p'']\bot \tau''$, then $p''$ has to be on the positive half of one of the axis). Thus, we know that all the points of $q\in\partial P$ such that $d_P(z,q)=\epsilon$  are in the I quadrant. We take as $p$ and $p'$ the pair of such a points for which the angle $\angle( pzp')$ is maximal, and the curve $\gamma_R$ is the part of $\partial P$ connecting $p$ and $p'$, as before.

We have already shown that for every interior point of $\gamma_R$ the point $z$ is the only most close point of $S_\epsilon$.

It remains to show that $z\notin m_\epsilon(q)$ for $q\in\partial P\setminus \gamma_R$. Indeed, if $z\in S_\epsilon(\tau'')$ for some side $\tau''$, then there exists $p''\in \tau''$ such that $d(z,p'')=\epsilon$, so $p''\in \gamma_R$, i.e. $\tau''\cap \gamma_R\neq\emptyset$. This means that $d(z,N_\epsilon(\tau^*))>0$  for any side $\tau^*$  disjoint with $\gamma_R$, and we again can find so small disc $D(z,r)$ such that $S_\epsilon\supset D(z,r)\cap ((S_\epsilon(\tau)\setminus D(p',\epsilon))\cup (S_\epsilon(\tau')\cap D(p,\epsilon)))$. Therefore, in the case when neither $[z,p]\bot\tau$, nor $[z,p']\bot\tau'$ we see that $S_\epsilon\cap D(z,r)$ is a union of two arcs ending in $z$, which are perpendicular (as smooth curves) to $[z,p]$ and $[z,p']$ respectively (see Figure 23).

Hence, as the last part of the shortest curve connecting $q$ to $z$ inside $\overline{P}$ is a segment, if it has an angle to one of the arcs less than $90^\circ$ there is a point $z'$ on the arc, and a curve connecting it to $q$ which is shorter. The only way for a most short curve joining $q$ with $S_\epsilon$ to end in $z$ is for its last segment to be contained in the angle $\angle pzp'$. But, as $q$ does not belong to the region bounded by $\gamma_R\cup [z,p]\cup [z,p']$, the shortest way from $q$ to $z$ can not pass inside the region.

This leaves us to consider the cases when $[z,p]\subset \gamma$ or $[z,p']\subset \gamma$, which can be treated in exactly the same manner as before (both in the cases of the angle $270^\circ$ and $360^\circ$).
\end{proof}

\begin{corollary}\label{structure.distsegment} The interior distances between points of $\partial P$ and the points of $m_\epsilon(p)$ coincide with the
Euclidean distances and are realized by a line segment contained in the interior of $P$, except from one of its ends.
\end{corollary}
\begin{proof} Let $p\in \partial P$ and $z\in m_\epsilon(p)$. Suppose that $z\notin m_\epsilon(p')$ for any $p\neq p'\in\partial P$. Since $z\in S_\epsilon$, there exists $p^*(z)=p^*\in \partial P$ for which $d(z,p^*)=\epsilon$, while by the definition of $S_\epsilon$ we know $d(p^*,S_\epsilon)=\epsilon$. So $z\in m_\epsilon(p^*)$. Hence $p^*=p$ and $d(p,z)=\epsilon=d(z,\partial P)$, thus $(p,z]\subset P$.

If $z\in m_\epsilon(p)\cap m_\epsilon(p'')$ for two different points of $\partial P$ we use the information from the previous lemma.
\end{proof}

\begin{corollary}\label{structure.nonintersectingconnections}The segments realizing the shortest distances from the
points of $\partial P$ to $S_\epsilon$ do not intersect except
their ends.
\end{corollary}
\begin{proof} Let $[z,p]$, respectively $[z',p']$, realize the shortest distance from $p$, respectively $p'$, to $S_\epsilon$.

If $[z,p]\cap [z',p']=z\neq z'$, then $d(z,p')<d(z',p')$ and $[z',p']$ is not a shortest path from $p'$ to $S_\epsilon$. Similarly it can not be that $[z,p]\cap [z',p']=z'\neq z$, $[z,p]\cap [z',p']=p\neq p'$, or $[z,p]\cap [z',p']=p'\neq p$.

Suppose $[z,p]\cap[z',p']=q$ for some $q\neq z,p,z',p'$ where $z\in m_\epsilon(p)$ and $z'\in m_\epsilon(p')$. Wlog we may assume $d(q,z)\geq d(q,z')$. Then the length of the curve $\gamma=[p,q]\cup [q,z']$ does not exceed $d(z,p)$ and $\gamma$ joins $p$ with $S_\epsilon$. As $\gamma$ is not a straight line and $\gamma\setminus \{p\}\subset P$, we can find a shorter curve connecting $p$ and $z'\in S_\epsilon$ inside $P$, i.e. $d_P(p,S_\epsilon)<d(p,z)$ and $z\not\in m_\epsilon(p)$. This contradiction proves the corollary.
\end{proof}

\begin{lemma}\label{structure.cornercase} Let $p\in \partial P$ and
$z,z'\in m_\epsilon(p)$ for some $z\neq z'$. Then $p$ is a vertex of $P$ and the angle at $p$ is at least $270^\circ$. Moreover, $m_\ve(p)$ is a connected arc of a circle of radius $\epsilon$.
\end{lemma}
\begin{proof} Assume that $d(p,z)>\epsilon$. Then
there exists $p'\in \partial P$, such that $d(p',z)=\epsilon$. This means that there are two different points $p,p'\in\partial P$ for which $z\in m_\epsilon(p)\cap m_\epsilon(p')$.
Applying Lemma \ref{structur.closesttomany} we get $p\in
\tilde{\phi}[t_0,t_1]$ and $p\neq \tilde{\phi}(t_0),\tilde{\phi}(t_1)$ because $d(p,z)>\epsilon$. But then  $S_\epsilon(p)=\{z\}$. This
contradiction shows that $d(p,z)=d(p,z')=\epsilon$.

If $p$ were an interior point of a side $\tau\subset\partial P$, then there exist exactly
two points outside $N_\epsilon(\tau)$ at the distance $\epsilon$ from $p$. Only one of them can be connected to $p$ by a segment described in the Corollary \ref{structure.distsegment}. Any other
point of $S_\epsilon$ will have the distance to $p$ greater than
$\epsilon$. This means that there is only one closest to $p$ point
of $S_\epsilon$ and the contradiction shows that $p$ has to be an endpoint of a side, i.e. a vertex.

Let $p$ is an endpoint point of a side $\tau$. We already  know that $d(p, S_\epsilon)=\epsilon$. This means that $m_\epsilon(p)$ is contained in the semicircle with the
centra in $p$ which is a part of $S_\epsilon(\tau)$. Let us denote this semicircle by $\gamma_p$.

If the angle at $p$ is $90^\circ$ the line segment which connects $p$ to $z\in S_\epsilon(p)$  by Corollary \ref{structure.distsegment} forms an acute angle with both sides of $P$ ending in $p$. Thus there are points of $\partial P$ which are closer to $z$ than $p$. This contradicts $d(z,p)=\epsilon$. Thus the angle at $p$ is at least $270^\circ$.

Since $S_\epsilon=\big(\cup(S_\epsilon(\tau')\setminus(\bigcup\limits_{\tau''\neq \tau'}N_\epsilon(\tau'')))\big)$, where the first union is taken over all  sides of $P$, we have $m_\epsilon(p)=\gamma_p\setminus (\bigcup\limits_{\tau''\neq \tau'}N_\epsilon(\tau''))=\gamma_p\setminus (\partial P + D(0,\epsilon))$.
Let $\mathcal{Q}\subset \partial P$ be a countable dense subset. Then $\partial P + D(0,\epsilon)=\bigcup\limits_{q\in \mathcal{Q}}D(q,\epsilon)$. We will show that $\gamma_p\setminus \bigcup\limits_{q\in \mathcal{Q}}D(q,\epsilon)$ is a connected arc.

Let us first consider the case when the angle at $p$ is $270^\circ$. Denote the
other side of $P$ ending in $p$ by $\nu$ (as we consider the covering of $P$ this side is uniquely defined).
W.l.o g. we may assume that
$p$ is in the origin of cartesian coordinates, $\tau$ goes along the positive $x$-axis, and
$\nu$ goes along the positive $y$-axis. Then $\gamma_p\setminus
N_\epsilon(\nu)$ is contained in a quart-circle $C$ in the third
quadrant. What part of this quart-circle belongs to $S_\epsilon$
depends from what is left after we take away all the discs $D(q,\epsilon),q\in\mathcal{Q}$.
Notice that if $q\in \partial P$ is a point of the third quadrant, then $q\notin Q(p,P)$ and $D(q,\epsilon)\cap \gamma_p=\emptyset$
(by property (vii) of $P\in VH(.,\frac12)$)). On the other hand, if the point $q$ is outside of the third quadrant then $D(q,\epsilon)\cap \gamma_p$ is a connected arc containing one of ends of $C$.
This means that the union of all intersections $D(q,\epsilon)\cap \gamma_p$ forms
two arcs containing the corresponding ends of $\gamma_p$ and what left after removing
of these arcs is a connected arc.

If the angle at $p$ is $360^\circ$ we may w.l.o.g. assume that the sides
go along the positive part of the $x$-axis. Then similar
reasoning can be applied with replacing the quart-circle by the
semicircle for $\gamma_p$ and the third quadrant by the left half-plane.
\end{proof}

We we are going to define now the continuous function $\widetilde{\psi}:[0,4n]\rightarrow S_\epsilon$, such that $\widetilde{\psi}(t)\in m_\epsilon(\widetilde{\phi}(t))$.
This condition defines $\widetilde{\psi}$ uniquely on the
intervals $(2k,2k+1)$, because,  by Lemma \ref{structure.cornercase}, all points of the sides of $P$ except vertices have exactly one closest point of $S_\epsilon$.

\begin{lemma}\label{psi:side} The function $\widetilde{\psi}$ is continuous on $(2k,2k+1)$.
\end{lemma}
\begin{proof} Let and $t_j\in (2k,2k+1)$, and $\lim t_j= t^*$, $t^*\in (2k,2k+1)$. Then $\widetilde{\phi}(t_j)
\rightarrow \widetilde{\phi}(t^*)$, and so $d(\widetilde{\phi}(t_j),\widetilde{\psi}(t_j)) =
d(\widetilde{\phi}(t_j), S_\epsilon)\rightarrow d(\widetilde{\phi}(t^*), S_\epsilon)$. Consider a limit point $z$ of the sequence $\widetilde{\psi}(t_j)$.
For this point holds $d(\widetilde{\phi}(t^*),z)=\lim d(\widetilde{\phi}(t_j),\widetilde{\psi}(t_j))
=d(\widetilde{\phi}(t^*),S_\epsilon)$ and $z\in S_\epsilon$, as the set $S_\epsilon$ is closed.

As $\widetilde{\psi}(t^*)$ is the unique point of $S_\epsilon$, which is at the distance $d(\widetilde{\psi}(t^*),S_\epsilon)$ of $\widetilde{\psi}(t^*)$,  it has to be $z=\widetilde{\psi}(t^*)$.
Thus, the sequence $\widetilde{\psi}(t_j)$ of points of the compact set $S_\epsilon$ has a unique limit point $\widetilde{\psi}(t^*)$, i.e. $\widetilde{\psi}$ is continuous in $t^*$.
\end{proof}

Let us now consider the vertices of $P$. If the angle at a vertex $p$ is less than $270^\circ$,
then the Lemma \ref{structure.cornercase} shows that the closest to $p$ point of $S_\epsilon$ is unique, and the same proof as in
Lemma \ref{psi:side} shows that setting $\widetilde{\psi}(t)$ for $t\in \widetilde{\phi}^{-1}(p)$ to the unique point of $m_\epsilon(p)$ we get the function $\widetilde{\psi}$ to be continuous in a neighborhood of $\widetilde{\phi}^{-1}(p)$.

It remains to consider the cases when the angle at $p$ is $270^\circ$ or $360^\circ$. In both cases
the set $m_\epsilon(p)$ is either a point or a (connected) arc by Lemma \ref{structure.cornercase}.
If the arc has degenerated to a point, then the above argument show that setting $\widetilde{\psi}$
to that point for $\widetilde{\phi}^{-1}(p)$ will define $\widetilde{\psi}$ continuously on $\widetilde{\phi}^{-1}(p)$ and in
its neighborhood. Non-degenerated case is described by the following lemma.

\begin{lemma}\label{psi:arc} If $p=\widetilde{\phi}([2k+1,2k+2])$ is a corner of $P$ with the intern angle at least
$270^\circ$ such that $m_\epsilon(p)$ is an arc $\gamma_p$, then
$\lim\limits_{t\rightarrow (2k+1)^-}\psi(t)$ and $\lim\limits_{t\rightarrow (2k+2)^+}\psi(t)$ exist and
are the endpoints of the arc $\gamma_p$. $\lim\limits_{t\rightarrow (2k+1)^-}\psi(t)\neq\lim\limits_{t\rightarrow (2k+2)^+}\psi(t)$.
\end{lemma}

\begin{proof} Let $p=\widetilde{\phi}([2k+1,2k+2])$. Consider first the case of $270^\circ$.

We place the origin of the coordinate system to $p$, and choose the direction of axes so that the intern angle at $p$ contains II, III and IV
quadrants, and so that the side $\widetilde{\phi}([2k,2k+1])$ goes along positive $x$-axis and $\widetilde{\phi}([2k+2,2k+3])$ - along positive $y$-axis (see Figure 24). Let the upper end of the curve $\gamma_p=m_\epsilon(p)$ is $r=(x_r,y_r)$. Consider a point $q=\widetilde{\phi}(t), t\in (2k+2,2k+3)$. The point $q$ is on the positive $y$-axis, and has coordinates $q=(0,y_q)$. Notice that $y_r\leq 0$ and $y_q>0$.

As $q$ is not a vertex, $\widetilde{\psi}(t)$ is the unique point of $S_\epsilon$ most close to $q$, and thus is
at the distance at most $d(q,r)$ from $q$. At the same time, it should be in $S_\epsilon$, and thus $\widetilde{\psi}(t)\not \in D(p,\epsilon)$. This means that $\widetilde{\psi}(t)$ has $y$-coordinate not less than the one of $r$, as any point of $(\mathbb{R}\times (y_r,-\infty))\setminus D(p,\epsilon)$ has distance to $q$ greater than $d(y,q)$.

Now, consider $(t_n)\subset (2k+2,2k+3)$, such that $t_n\rightarrow 2k+2$. Then $\widetilde{\phi}(t_n)\rightarrow p$ along $y$-axis. $d(\widetilde{\psi}(t_n),\widetilde{\phi}(t_n))\leq d(r,\widetilde{\phi}(t_n))$. For any limit point $z$ of $(\widetilde{\psi}(t_n))$ holds, thus, $d(p,z)\leq \epsilon$. As $z\in S_\epsilon$ implies $d(z,p)\geq\epsilon$, we see that $d(p,z)=\epsilon$.

By Lemma \ref{structure.cornercase}, $z\in \gamma_p$. At the same time the $y$-coordinate of any limit
point should be no less than $y$-coordinate of $r$. This means that $z=r$. As the sequence is bounded, this means that the limit $\lim  \widetilde{\psi}(t_n)$ exists and is $r$. The side $\widetilde{\phi}([2k,2k+1])$ can be dealt with in the same manner.

Now, let the angle is $360^\circ$. Then we can place the coordinate system so that the origin is at $p$ and the sides attached to it go along the positive $y$-axis. Consider again $\widetilde{\phi}(t_n)\rightarrow p$, such that $t_n\in (2k+2,2k+3)$ (and so belong to the side which is "attached" to the interior of $P$ in the second quadrant). By Corollary \ref{structure.distsegment}, the interior distance between $\widetilde{\phi}(t_n)$ and $\widetilde{\psi}(t_n)$ is the same as the Euclidean distance. This means that all $\widetilde{\psi}(t_n)$ are in the left half-plane and so are the limit points of the sequence. As the bounded sequence  $(\widetilde{\psi}(t_n))$ has at least one limit point and this point belongs the the arc $\gamma_p=m_\epsilon(p)$, the arc has some points in the closed left half-plane. Consider the highest of those points and repeat the argument we have used for the angle of $270^\circ$. The side $\widetilde{\phi}([2k,2k+1])$ can be dealt with in the same manner.
\end{proof}

Taking into the account Lemmas \ref{structure.cornercase} and \ref{psi:arc} we can use $\widetilde{\phi}^{-1}(p)$ to parameterize the arc $m_\epsilon(p)$ in such a way that $\widetilde{\psi}$ is continuous in all points of $\widetilde{\phi}^{-1}(p)$ (including the ends). This completes the construction of a continuous function $\widetilde{\psi}:[0,4n]\mapsto S_\epsilon$. As every point of $S_\epsilon$ is at the distance $\epsilon$ of some $p\in\partial P$, it is in $m_\epsilon(p)$ for some $p\in\partial P$. Thus $\widetilde{\psi}([0,4n])=S_\epsilon$, which shows that $S_\epsilon$ is a curve, may be not simple (notice that we keep the assumption that $\widetilde{\psi}(0)=\widetilde{\psi}(4n)$).

In order to show that the curve $S_\epsilon$ is simple, it is enough to show that for any $q\in S_\epsilon$ the set $\widetilde{\psi}^{-1}(q)$ is a closed interval (or a point). Then, by factorizing $[0,4n]$ by those intervals we obtain a simple parametrization of $S_\epsilon$.

By the construction, $\widetilde{\psi}^{-1}(q)$ can be more than one point only in two cases: there are $p,p'\in\partial P$ such that $q\in m_\epsilon(p)$ and $q\in m_\epsilon(p')$, or there is a vertex $p$ such that $m_\epsilon(p)=\{q\}$ (and $q\not\in m_\epsilon(p')$ for $p\neq p'\in\partial P$). In the first case the Lemma \ref{structur.closesttomany} gives us the desired property, and the second case is trivial.

Now, when we know that $S_\epsilon$ is a simple curve, Lemma \ref{structure:arcline} shows that the curve $S_\epsilon$ is rectifiable.

\subsubsection{Lipschitz property of $S_\epsilon$}

\begin{definition} We say that a simple curve $\gamma$ is $\eta$-Lipschitz, if a distance between any two of its points measured along the curve (if the curve is closed we take the shortest of the two pathes along $\gamma$ connecting the points) is at most $\eta$ times the Euclidean distance between the points. We say that a curve is $\mu$-locally $\eta$-Lipschitz if the above holds for points at Euclidean distance closer then $\mu$.
\end{definition}

Note that an $\eta$-Lipschitz curve is necessarily simple.

In this section we will show that $S_\epsilon$ is not only a simple closed curve (as it was shown above), but also an $\frac{\epsilon}4$-local $\eta$-Lipschitz curve for some universal constant $\eta>0$.

\begin{lemma}\label{distance.two.S} Given $r>0$ such that $r+\epsilon<1/10$, for every point $x\in S_\epsilon$ there exists a point $y\in S_{\epsilon+r}$, such that $d(x,y)<2 r$.
\end{lemma}
\begin{proof} Let $x\in S_{\epsilon}$, $p\in\partial P$ satisfies $d(x,p)=\epsilon$ and let $m_\epsilon(p)=\{x\}$.
Consider $y\in m_{\epsilon+r}(p)$.

If $\{y\}\not\in m_{\epsilon+r}(q)$ for $q\neq p$ then (as it has been shown) $d(p,y)=\epsilon+r$. By continuity of the distance to $\partial P$, there exists a point $x'\in S_\epsilon\cap (p,y)$. Clearly $d(x',\partial P)\leq d(x',p)$, so $d(x',p)\geq\epsilon$ and $d(x',y)\leq r$. Then $d(y,\partial P)\leq d(x',\partial P)+d(x',y)\leq \epsilon +r$ and
$y\in S_{\epsilon+r}$ implies that all the inequalities in the estimates must be equalities. Thus $d(x',p)=\epsilon$ and $x'=x$. But then $d(x',y)=r$.

Let $y\in m_{\epsilon+r}(q)$ for some $q\neq p$. Then by lemma \ref{structur.closesttomany} there are $t_0,t_1\in [0,4n]$ such that $p\in\widetilde{\phi}([t_0,t_1])$, $y\in m_{\epsilon+r}(z)$ for all $z\in\widetilde{\phi}([t_0,t_1])$ and $d(y,\widetilde{t_0})=d(y,\widetilde{t_1})=\epsilon+r$. Let $p_0=\widetilde{\phi}(t_0)$ and $p_1=\widetilde{\phi}(t_1)$. Arguing as above,we derive that $S_\epsilon\cap(y,p_0)=\{x_0\}$, $S_\epsilon\cap (y,p_1)=\{x_1\}$, and $d(y,x_0)=d(y,x_1)=r$. By continuity of the distance, $S_\epsilon\cap (y,p)\neq\emptyset$. The interval $(y,p)$ is contained in the polygon $Q$ bounded by $[y,p_0]$, $\widetilde{\phi}([t_0,t_1])$, and $[y,p_1]$. Hence $S_\epsilon\cap Q\neq \emptyset$. As we already know that $S_\epsilon$ is a simple curve, it follows that $S_\epsilon\cap Q$ is a curve $\gamma$ connecting $x_0$ and $x_1$ (and has no other points of intersection with $\partial Q$).

By Lemma \ref{structure.nonintersectingconnections} the segment $[p,x]$ does not intersect neither $[p_0,x_0]$ nor $[p_1,x_1]$, unless $x=x_0$, $x=x_1$, $p=p_0$, or $p=p_1$ , because they all are segments connecting points of $\partial P$ to its closest points of $S_\epsilon$. If $x=x_0$ or $x=x_1$ then the distance from $d(x,y)=r$, and as $m_\epsilon(p)=\{x\}$ the case $p=p_0$ implies $x=x_0$ and $p=p_1$ implies $x=x_1$.

Also $[p,x]\cap [y,x_0)=\emptyset$ and $[p,x]\cap [y,x_1)=\emptyset$ because $d(q,\partial P)\leq\epsilon$ for every $q\in [x,p]$.
Since, obviously, $[x,p)\cap\widetilde{\phi}([t_0,t_1])=\emptyset$, it follows that $[x,p)\cap\partial Q=\emptyset$ and $x\in\gamma=Q\cap S_\epsilon$ (as soon as $p\neq p_0$, $p\neq p_1$). The proof of the Lemma \ref{structur.closesttomany} describes the shape of $\partial P\cap \overline{Q}$, and this yields that $S_\epsilon\cap Q$ is contained in the part of $Q$ cut by the horizontal/vertical lines passing through $x_0$ and $x_1$ (see Figure 25). That parts of $Q$ have diameters less than $2r$, so  $d(x,y)<2r$. Notice that the argument above doesn't use that $m_\epsilon(p)=\{x\}$.

Finally, let $m_\epsilon(p)\neq \{x\}$ (i.e. $m_\epsilon(p)$ contain points other than $x$) and $y\not\in m_{r+\epsilon}(q)$ for $q\neq p$. Then, Lemma \ref{structure.cornercase} implies that $p$ is a vertex of $P$.
Denote the connected simple arc $m_{\epsilon}(p)$ by $\gamma_p\subset S_{\epsilon}$. The set $m_{\epsilon+r}(p)$ may be an arc or a point, of which we will consider only the first as the later is only degenerated case of the former.

Let $y'$ and $y''$ are the endpoints of the arc $m_{\epsilon+r}(p)$. Denote the sides adjoint to $p$ as $\tau'$ and $\tau''$. The points $y'$ and $y''$ may or may not be endpoints of the arc $\overline{D}(p,\epsilon+r)\setminus (N_{\epsilon+r}(\tau')\cup N_{\epsilon+r}(\tau''))$ (see Figure 26). We consider the case when $y''$ is and $y'$ is not an endpoint, as all the other cases can be treated by a similar argument.

As $y'$ is not an end of $\overline{D}(p,\epsilon+r)\setminus (N_{\epsilon+r}(\tau')\cup N_{\epsilon+r}(\tau''))$, there exists $p'\in \partial P\setminus(\tau\cup\tau')$ such that $d(y',p')=\epsilon+r$. Lemma \ref{structur.closesttomany} implies that in a subcurve $\gamma_p$ of $\partial P$ connecting $p$ and $p'$ all points have $y'$ as the closest point of $S_{\epsilon+r}$ and we can repeat the argument of the previous case for all points of $S_{\epsilon}$ in the region $Q$ bounded by $\gamma_p$, $[y',p]$, and $[y',p']$ (see Figure 27). On the other hand, for the region $Q^*$, bounded by $m_{\epsilon+r}(p)$, $[y',p]$, and $[y'',p]$, the set $S_\epsilon\cap Q^*$ is obviously an arc of the circle of radius $\epsilon$, thus every point of it is within distance $r$ from $m_{\epsilon+r}(p)\subset S_{\epsilon+r}$ (see Figure 28).

To complete the proof of this case it suffice that we show that $m_\epsilon(p)\subset Q\cup Q^*$. Indeed denote the point of $[y',p']\cap S_\epsilon$ as $x'$ and the point of $[y'',p]\cap S_\epsilon$ as $x''$. The arc $m_\epsilon(p)$ can not extend further than $x''$, as it should be a subset of $\overline{D}(p,\epsilon)\setminus (N_\epsilon(\tau')\cup N_\epsilon(\tau''))$. Also, as by triangle inequality $d(y',x')+d(x',p)>\epsilon+r$, and $d(y',x')=r$, we know that $d(x',p)>\epsilon$, i.e. $x'\not\in m_\epsilon(p)$. By continuity of the distance applied to $[y',p]$, we see that there are points of $m_\epsilon(p)$ inside the region $Q\cup Q^*$, and the curve $m_\epsilon(p)$ does not intersect $\partial (Q\cup Q^*)$, except probably by its end point. This means that $m_\epsilon(p)\subset Q\cup Q^*$, which completes the proof of our last case.
\end{proof}

\begin{lemma} For any $x,y\in S_\epsilon$, such that $d(x,y)=r<\frac12\varepsilon$, where $r+\epsilon<1/10$, one of the two components of $S_\epsilon\setminus\{x,y\}$ is contained in a ball of radius not greater than $3r$.
\end{lemma}
\begin{proof} Let $p_x,p_y\in\partial P$ satisfy $d(x,p_x)=d(y,p_y)= \epsilon$. Then all the points of $[x,p_x]$ and $[y,p_y]$ have distance to $\partial P$ not greater than $\epsilon$.

By the triangle  inequality, $\epsilon-r/2\leq d(z,\partial P)\leq \epsilon+r/2$ for every $z\in[x,y]$. The curve $\gamma=[p_x,x]\cup [x,y]\cup [y,p_y]$ divides $P$ on two parts $P_1$ and $P_2$. One of them, say $P_1$, is disjoint with $S_{\epsilon+\frac{\sqrt{3}}2r}$ because $\gamma \cap S_{\epsilon+\frac{\sqrt{3}}2r}=\emptyset$.

Let $\widetilde{\gamma}=\partial P\cap\overline{P_1}$. We have already proven that $\gamma'=\widetilde{\psi}(\widetilde{\phi}^{-1}(\tilde{\gamma}))$ is a simple curve (we consider $\widetilde{\gamma}$ on the covering of $P$).  Clearly $x,y\in\gamma'$. Denote by $\gamma_{xy}$ the part of $\gamma'$  connecting $x$ and $y$. Then $\gamma_{xy}$ is a sub-curve of the simple curve $S_\epsilon$. We will show first that $d(z,[x,y])\leq 2r$ for every $z\in\gamma_{xy}$.

Indeed, if $z \in \gamma_{xy}\setminus P_1$ we consider $p\in \partial P$ such that $d(p,z)=\epsilon$ (and so $(p,z]\subset P$). Then there is $t$ such that $\phi(t)=p$ and $\psi(t)=z$. By the choice of $\gamma_{xy}$, $t\in \phi^{-1}(\partial P\cap \partial P_1$ so $p\in \partial P\cap \partial P_1\setminus \{p_x,p_y\}$. The segment $(z,p)$ intersects thus $\gamma$, but by Corollary \ref{structure.nonintersectingconnections} it can not intersect neither $[x,p_x]$ nor $[y,p_y]$. Then there for $s\in (x,y)\cap (z,p)$ holds $d(z,s)=d(z,p)-d(s,p)\leq\epsilon-(\epsilon-
r/2)=r/2$. Hence $\gamma_{xy}\setminus P_1\subset ([x,y])_{\frac12 r}$.

If, on the other hand, $z\in P_1\cap \gamma_{xy}$ then, since $z\in S_\epsilon$,
it satisfies, by the Lemma \ref{distance.two.S}, $d(z,S_{\epsilon+\frac{\sqrt{3}r}2})\leq \sqrt{3} r$.
But $d(z,P\setminus P_1)>\sqrt{3} r$  for every $z\in P_1\setminus (D(p_x,\epsilon)\cup D(p_y,\epsilon)\cup ([x,y])_{2 r}$. We see that $\gamma_{xy}\cap P_1\subset([x,y])_{2 r}$, as $S_\epsilon\cap (D(p_x,\epsilon)\cup D(p_y,\epsilon))=\emptyset$.

In all cases we get  $\gamma_{xy}\subset [x,y]+D(0,2r)\subset D(\frac{x+y}2,3r)$.
\end{proof}

\begin{lemma} If $r<\frac14\epsilon$ then the length of intersection $S_\epsilon\cap D(a,r)$, does not exceed $20r$.
\end{lemma}
\begin{proof} Recall that $S_\epsilon$ consists  of horizontal and vertical segments, and arcs of circles of radius $\epsilon$.

Consider first only the part $V$ of $S_\epsilon\cap D(a,r)$ consisting of vertical segments. It follows from the construction that
if $z\in V$ is an interior point of a vertical segment then there exist a horizontal segment $[z,p]$ of length $\epsilon$ such that $p\in\partial P$  and $d(s,\partial P)<\epsilon$ for every $s\in (z,p)$. Hence $(z,p)\cap S_\epsilon=\emptyset$.
This means that any horizontal line intersects $V$ in at most two points.
Hence the total length of $V$ does not exceed twice the  length of the projection of $V$
onto the vertical axis which is bounded by $2r$ because $V\subset D(a,r)$. That gives an estimate $4r$ for the length of $V$.
The case of the horizontal counterpart is similar.

The part $C\subset S_\epsilon\cap D(a,r)$ consisting of circular arcs decomposes onto parts $C_1,C_2,C_3$ and $C_4$ contained respectively subarcs of the arcs in the first, second, third, and forth quarters with respect to the coordinate system centered in the center of the corresponding circle. We project $C_1\cup C_3$ orthogonally onto the line $x+y=0$. For every point $z$ of an arc in $C_1\cup C_3$ which centers in $p$, there exists an open segment parallel to $x-y=0$ which has an end in $z$, and length $\sqrt{2}\epsilon$ such that it lies in $D(p,\epsilon)$ and, thus, does not intersect $S_\epsilon$. An argument similar to the one  above proves that  every point of the line is the image of at most two points of $C_1\cup C_3$. Since the projection of the arcs is $\sqrt{2}$ bi-Lipschitz, the total length of $C_1\cup C_3$ is bounded by $4\sqrt{2}r$. The same argument works for $C_2\cup C_4$.

All together, the parts of $S_\epsilon\cap D(a,r)$ have total length bounded by $(8+8\sqrt{2})r<20r$.
\end{proof}

Combining the above two lemmas  we get that $S_\epsilon$ is $\frac1{12}\epsilon$-locally $\eta$-Lipschitz, for $\eta=60$.

\subsubsection{Decomposition of $P$}

Before we conduct the decomposition of $P$ we need the following simple lemma.

\begin{lemma}\label{partialP.Lip} When we consider the polygon as a manifold with respect to its internal metric, then $\partial P$ (as boundary with respect to the internal for $P$ metrics) is $\epsilon/4$-locally $\sqrt{2}$-Lipschitz.
\end{lemma}
\begin{proof} Let $x,y\in \partial P$ and $d_P(x,y)<\epsilon/4$. Consider a curve which realizes this distance $\gamma\subset \overline{P}$. This curve consists of points of $\partial P$. If two points from the same side of $P$ belong to $\gamma$, then the whole interval between them is a subinterval of the side, as this is the shortest path between the two points. Thus, at most two ends of the open intervals can belong to the same side of $P$. As $P$ has finitely many sides, this means $\gamma$ contains at most finitely many open intervals of $P$ and the rest is finitely many subintervals or points of $\partial P$. It suffices thus show that every open interval of $\gamma\cap P$ (of length less than $\epsilon/4<\frac12$) can be replaced by a curve in $\partial P$ with the same ends, and the length not more than $\sqrt{2}$ as large. This argument is very similar to the one in the proof of Lemma \ref{structur.closesttomany}, and is left to the reader.
\end{proof}

To decompose $P$
we consider two curves $S_\epsilon$ and $S_{2\epsilon}$ where $\epsilon=\frac1{32}$ (remember that $P$ is re-scaled to keep $d=1$).

Since $S_\epsilon$ and $S_{2\epsilon}$ are simple curves,  each of them is a boundary of a connected simply connected domain. Thus $P$ is divided onto three parts: $\Omega_1$ bounded by $\partial P$ and $S_\epsilon$; $\Omega_2$ bounded by $S_\epsilon$ and $S_{2\epsilon}$; $\Omega_3$ bounded by $S_{2\epsilon}$.

We will decompose each part separately, though decompositions of $\Omega_2$ and $\Omega_3$ will blend on the boundary.

We begin with  $\Omega_1$. In the description bellow we use the notion of cut. By
{\it a cut} we mean a segment or a union of two segments with common ends such that one of its endpoints is contained in $\partial P\cap\partial\Omega$ and another in $S_\epsilon$. Connected components of $\Omega_1$ with all the cuts removed are elements of the desired decomposition (and due to the first condition every connected component of $\Omega\setminus P$ is attached to only one element of decomposition of $\Omega_1$). We call a cut $\kappa$ {\it regular} provided that for any pair of points $x,y$ such that $x\in \kappa, y\in \partial P\cup S_\epsilon$, the Euclidian distance $d(x,y)$ is not more than $\tilde{\eta}$ times as short ($\tilde{\eta}>0$ is a universal constant) as a most short curve in $\kappa\cup \partial P\cup S_\epsilon$ connecting $x$ and $y$, as soon as the internal for $P$ distance between $x$ and $y$ is at most $\epsilon/{36}$.

First we describe cuts with one end on a side $\tau$ adjoint to two angles of $90^{\circ}$ (see Figure 29). Note that by condition (v) of $P\in VH(.,\frac12)$ every such side has length at least $\frac12$.

Since $S_\epsilon\supset S_\epsilon(\tau)\setminus (\partial P +D(0,\epsilon))$, by the condition (v) of $P\in VH(.,\frac12)$  there is an interval $I_\tau\subset S_\epsilon$, which is parallel to $\tau$ such that $d(\tau, I_\tau)=\epsilon$, and the image of $I_\epsilon$ projected  orthogonally on $\tau$ covers all points of $\tau$ with distance from its endpoints greater then $\epsilon$.

Before we removed the sides of length zero, in order to study $S_\epsilon$, the polygon $P$ was in the class $VH(3,\frac12)$. Thus, there are no subinterval of $\tau$ of length greater than $6$ disjoint with $\partial\Omega$.
If $|\tau|\geq 6+4\epsilon$, than there is $p\in\tau\cup\partial\Omega$ and a point $z\in I_\tau$ such that $[p,z]\bot \tau$ and
$d(p,\partial\tau)>2\epsilon$. Let $[p,z]$ be a cut. If $|\tau|>6k+4\epsilon$ for some integer $k$, we make $k$ such cuts dividing $\tau$ into pieces of length greater than $\epsilon$ and less than $6+4\epsilon$ (see Figure 29). It is clear that such cuts are regular.

If the length of $\tau$ is less then $6+4\epsilon$ there still exists $p\in\tau\cap\partial \Omega$. Let $z$ is middle of the interval $I_\tau$. We show that the cut $[p,z]$ is regular. As $z$ is the middle of $I_\tau$, $|I_\tau|>\frac12- 2\epsilon=14\epsilon$ it is clear that if $x\in S_\epsilon$ and $d(x,[p,z])<\epsilon/36$ then $x\in I_\tau$. As $|\tau|< 196\epsilon$ and $d(I_\tau,\tau)=\epsilon$, it is clear that there exists a universal constant $\tilde{\eta}$ such that $d(x,z)+d(z,y)<\eta d(x,y)$, for $x\in S_\epsilon, y\in [p,z]$, when $d_P(x,y)\leq \epsilon/36$.

Notice that either $d(p,\partial \tau)>\epsilon/36$ or the angle between $[p,z]$ and $\tau$ is less than $\frac{\pi}3$ (see Figure 30). Let $x\in \partial P$, $y\in [p,z]$ and $d_P(x,y)\leq \epsilon/36$. If $d(p,\partial \tau)>\epsilon/36$ then, as $Q(\tau,P)\cap \partial P=\emptyset$, it must be that $x\in\tau$, and the case follows to the same argument as the previous one.

On the other hand if $x\not\in \tau$ then it has to be outside of $Q(\tau,P)\cup\tau$. In such a case (see Figure 31) $d(p,\partial\tau)<\epsilon/36$ and $d_P(y,x)\geq \frac{2}{\sqrt{3}} d_P(p,y)$, so $d_P(p,x)<d_P(p,y)+d_P(y,x)<\epsilon/12$. By Lemma \ref{partialP.Lip} there exists $\gamma_1(p,x)\subset \partial P$ such that $|\gamma_1|\leq \sqrt{2}d_P(p,x)\leq \sqrt{2}(\frac{\sqrt{3}}{2}+1)d_P(y,x)<6 d_P(y,x)$. Thus, the curve $[y,p]\cup \gamma_1$ connects $x$ and $y$ inside $\partial P\cup [p,z]$, and $|[y,p]\cup \gamma_1|\leq 8 d_P(x,y)$, which completes the proof of the regularity of the cup $[p,z]$.

Consider now a vertex $p$ adjoint to an angle of $360^\circ$. As a matter of fact there are two non-zero and one zero length sides ending in $p$. Because of the zero length side, we know that $p\in \partial\Omega$.
We make a cut $[p,z]$ where $z\in S_\epsilon$ to be a continuation an the non-zero sides with an end in $p$ (see Figure 32). As $z\in m_\epsilon(p)$ for any $y\in [p,z]$ holds $d(y,z)=d(y,S_\epsilon)$. Thus, if $d(y,x)<\epsilon/36$, $x\in S_\epsilon$, then $d(y,z)<d(y,x)$ and by the triangle inequality $d(z,x)<2 d(y,x)$. Applying $\epsilon/12$-local $\eta$-Lipschitz property of $S_\epsilon$ we can find a curve $\gamma_1(z,x)\subset S_\epsilon$ such that $|\gamma_1|<2\eta d(x,y)$. The curve $\gamma(y,x)=[y,z]\cup \gamma_1$ has thus length at most $(2\eta+1)d(x,y)$, which proves regularity of the cut $[p,z]$ on that end of the cut. On the other hand, $[p,z]$ is orthogonal to a side of $Q(p,P)$, thus $d(y,p)=d(y,\partial P)$ for every $y\in [p,z]$ and a similar argument show regularity of the cut $[p,z]$ on its other end.

Let us consider a side $\tau$ adjoint to two angles of $270^\circ$. By the conditions (v) and (vii) of $P\in VH(.,\frac12)$ we conclude that there is $I_\tau\subset S_\epsilon$, which is parallel to $\tau$ and its orthogonal projection that covers all $\tau$ (see Figure 33). Thus we take $p\in \tau\cap \partial \Omega$ and make the cut $[p,z]\bot \tau$. The same argument as before assure that the cut is regular, though we have to modify the definition of a regular cut so that the estimate of the distance along the curve via Euclidian holds only till the next cut.

Observe now that after performing all the above cuts $\Omega_1$ is divided into such pieces that every side of $\partial P$ which belongs entirely to the closure of one of them, say $A$, is adjoint to exactly one angle of $90^\circ$ and one of $270^\circ$. This means that curves $\overline{A}\cap \partial P$ are $\sqrt{2}$-Lipschitz not only with respect to interior metrics of $P$ but also with respect to Euclidean metrics.

Now, we see that the boundary of every piece $A$ is an $\epsilon/36$-locally $\tilde{\eta}$-Lipschitz curve (if two points of $\partial A$ are closer  than $\epsilon/36$, then they are either both on $\partial P$ or both on $S_\epsilon$ or both on the  cut, or one  of them on the cut and another one on $S_\epsilon$ or $\partial P\cap \overline{A}$. In all the cases the required inequalities are satisfied according to what we have already demonstrated.

Let us observe now that if some piece $A$ has diameter less than $100$ than $A$ is
$M$-elementary for some constant $M=M(\tilde{\eta},\epsilon)$. Indeed, there exist an $\epsilon/36$ net $X\subset\partial A$ with cardinality $M'=M'(\epsilon)$. Any two elements
$x,y\in X$ can be joint by the chain of its elements $x=a_1,a_2,\dots,a_k=y$ such that
$d(a_j, a_{j+1})<\epsilon/36$ and then, from the locally Lipschitz property, by the curve
of the length smaller then $M''=M'\cdot \tilde{\eta}\epsilon$ contained in $\partial A$. Hence for any two points $s,w\in \partial A$  either $d(s,w)\leq\epsilon/36$ and then there is a curve in $\partial A$ shorter then $\tilde{\eta}d(s,w)$ joining them,
or $d(s,w)>\epsilon/36$ and then the curve in $\partial A$ joining $s$ with the closest element $x\in X$,
then joining $x$ with $y\in X$ which is the closest to $w$ element of $X$ finally joining $y$ with $w$ has total length not exceeding $(M''+2\tilde{\eta})\epsilon<M \epsilon/36<M d(s,w)$.
Hence $A$ is $M$ elementary for $M=36(M''+2\tilde{\eta})$.

It remains to deal with the pieces of diameter greater than $100$.

Let $A$ is such a piece. It follows from the construction  that $\overline{A}\cap\partial P$ is a broken line of vertical and horizontal segments of which each but the first and the last is adjoint to one angle of $90^\circ$ and one angle of $270^\circ$. W.l.o.g. we may assume that the position of the interior of $A$ is right - up with respect to this broken line, i.e. that $90^\circ$ angles are in the first quadrant (see Figure 34). Notice that due to the construction and the conditions (v) and (vi) of $P\in VH(.,\frac12)$ the curve $S_\epsilon\cap \partial A$ consists of vertical and horizontal intervals and the circle arcs which belong to the first quarter, except possibly for the first and last ones. Let $N$ be a line, which makes $45^\circ$ with the coordinate axes and pass through the second and forth quadrants. The orthogonal projection of both $S_\epsilon$ and $\partial P$ onto $N$ is bi-Lipschitz except for the $200\epsilon$-neighborhood of the the ends of the projection. Let us now choose a point of $\partial P\cap\partial\Omega$ which is projected at the distance of approximately $50$ (with the marginal of $10$) from the upper end of the projection of $A$. We make a cut at this point in the direction orthogonal to $N$. Due to the bi-Lipschitz property of the projection in $\epsilon$-neighborhood of the cut, the cut is regular and the upper components of $A\setminus \overline{py}$ has diameter less than $100$, thus being $M$-elementary, while the other has diameter
at least $30$ smaller then the diameter of $A$. Repeating this operation suitable number of times we cut $A$ into $M$-elementary pieces.

To decompose $\Omega_2\cup S_{2\epsilon}\cup \Omega_3$ we start by looking at all closed diadic squares with side $\epsilon/16$ which intersect $\Omega_3$. Denote the union of the squares by $\widetilde{\Omega}$. It is clear that $\Omega_3\subset\widetilde{\Omega}\subset \Omega_2\cup S_{2\epsilon}\cup \Omega_3\subset \Omega$. Let $\Omega^*$ is the complement of the infinite component of $\mathbb{R}^2\setminus \widetilde{\Omega}$. As $\Omega$ simply connected, $\Omega^*\subset\Omega$. Also $\partial\Omega^*\subset\partial\widetilde{\Omega}$ and so is within $\epsilon/8$ distance from $S_{2\epsilon}$. As each diadic square is $M$-elementary, it remains only to divide $\Omega_2\setminus\Omega^*$ into $M$-elementary pieces.

Observe, that $\partial \Omega^*$ is a simple curve which is $\epsilon/32$-locally $\sqrt{2}$- Lipschitz. Let $K$ is a collection of points on $\partial \Omega^*$ which are at the distance between $30$ to $40$ units from the next one, measured by the length of $\partial \Omega^*$. For a point $z\in K$ let $s\in S_\epsilon$ is such that $d(s,z)=d(S_\epsilon,z)$. Let $r$ is middle of $[s,z]$. Then choose $q\in \partial \Omega^*$ such that $d(q,r)=d(\partial \Omega^*, r)$. Notice, that $\frac78\epsilon<d(s,z)\leq\epsilon$ as $d(z,S_{2\epsilon})<\frac18\epsilon$. As well, $d(s,q)\geq d(s,\partial\Omega^*)\geq d(s,\partial S_{2\epsilon})-d(\partial \Omega^*,S_{2\epsilon})\geq \frac78\epsilon$. By the triangle inequality this implies $d(r,q)>\frac38\epsilon$, while $d(r,q)\leq d(r,z)\leq\frac12\epsilon$. This, in particular shows that the angle between two intervals of the cut $\kappa=[s,r]\cup[r,q]$ is obtuse.

We claim that the cut $\kappa(z)=[p,r]\cup[r,q]$ is a regular one. Indeed, if $x\in S_\epsilon, y\in \kappa$ and $d(x,y)<\epsilon/36$ then $y\in [s,r]$. As $s$ is the closest to $r$ point of $S_\epsilon$, $d(y,s)\leq d(y,x)$, and by the triangle inequality $d(x,s)\leq 2d(x,y)<\epsilon/12$. Now, by $\epsilon/12$-locally Lipschitz property of $S_\epsilon$ there exists a curve $\gamma_1(x,s)\subset S_\epsilon$, such that $|\gamma_1|\leq\eta d(x,s)$, thus $\gamma(x,y)=\gamma_1\cup[s,y]$ satisfies $|\gamma|\leq (2\eta+1)d(x,y)$. The other end of the cut can be treated by a similar argument.

After we remove all the cuts, $\Omega_1\setminus\Omega^*$ is divided in pieces of diameter less than $100$ with $\epsilon/36$-locally $\tilde{\eta}$-Lipschitz boundary. As before, those pieces are $M$-elementary.

This completes the decomposition of $P$ in $M$-elementary pieces, and thus completes the decomposition of $\Omega$.


\begin{thebibliography}{HD}
\normalsize
\baselineskip=17pt
\bibitem[ACPP]{ACPP} G. Alberti, M. Cs\"ornyei, A. Pe\l czy\'nski, D.Preiss. \emph{BV has the bounded approximation property.} J. Geom. Anal. 15 (2005), no. 1, 1--7.
\bibitem[AFP]{AFP} L. Abrosio, N. Fusco, D. Pallara. \emph{Functions of Bounded Variation and Free Discontinuity Problems} Oxford Mathematical Monographs, 2000
\bibitem[Mar]{Mar} O. Martio. \emph{John domains, bi-lipschitz balls and Poincare inequality.} Rev. Roumaine Math. Pures Appl. 33 (1988), no.1--2, 107--112
\bibitem[PW1]{PW1} A. Pelczynski, M. Wojciechowski \emph{Handbook of Banach the geometry of BS, Sobolev Spaces}
\bibitem[PW2]{PW2} A. Pelczynski, M. Wojciechowski \emph{Contribution to the isomorphic classification of Sobolev Spaces $L^p_k(\Omega) (1\leq p<\infty)$}; Recent Progress in Functional Analysis, Proceedings Valdivia Conference, Valencia, July 2000, K. D. Bierstedt, J. Bonet, M. Maestre and J, Schmets edt. North- Holland Math. Stud. (2001), 133-142
\bibitem[Pom]{Pom} Ch. Pommerenke. \emph{Boundary behaviour of conformal maps.} Fundamental principals of Mathematical Sciences, 299, Springer-Verlag, 1992
\bibitem[St]{St} Stein. \emph{Singular Integrals and Differentiable Properties of Functions}
\bibitem[Tuk]{Tuk} P. Tukia. \emph{The planar Schenflies teorem for Lipschitz maps.} Ann. Acad. Sci. Fenn. Ser. A I Math. 5 (1980), no. 1, 49--72.
\bibitem[W]{W} M. Wojciechowski \emph{Bounded approximation property of spaces of functions of bounded variation of arbitrary order}
\bibitem[Woj]{Woj} P. Wojtaschyk \emph{Banach Spaces for Analysts} Cambridge Studies in Advanced Mathematics, 25, Cambridge University Press, 1991
\bibitem[Z]{Z} W.P. Zeimer. \emph{Weakly differentiable functions.} Graduate Texts in Mathematics, 120, Springer-Verlag, 1989
\end{thebibliography}
\end{document}